\documentclass{amsart}
\usepackage[utf8]{inputenc}
\usepackage{graphicx,amssymb,amsmath,amsthm}
\usepackage{caption}
\usepackage{subcaption}
\usepackage{indentfirst}
\usepackage{cancel}
\usepackage{lipsum}
\usepackage{epstopdf}
\usepackage{epsfig}
\usepackage{enumerate,color}   
\usepackage[overload]{empheq}

\newcommand{\veps}{\varepsilon}
\numberwithin{equation}{section}
\usepackage{soul}
\everymath{\displaystyle}

\theoremstyle{plain}
\newtheorem{theorem}{Theorem}[section]
\newtheorem{lemma}[theorem]{Lemma}
\newtheorem{proposition}[theorem]{Proposition}

\theoremstyle{definition}

\newtheorem*{defi*}{Definition} 
\theoremstyle{remark}

\theoremstyle{remark}

\title[Solitary waves for Euler-Poisson system]{Small amplitude limit of solitary waves for  the Euler-Poisson system}

\author[J. Bae]{Junsik Bae}
\address[JB]{Department of Mathematical Sciences, Ulsan National Institute of Science and Technology, Ulsan, 44919, Korea}
\email{jsbae@unist.ac.kr}
 
\author[B. Kwon]{Bongsuk Kwon}
\address[BK]{Department of Mathematical Sciences, Ulsan National Institute of Science and Technology, Ulsan, 44919, Korea}
\email{bkwon@unist.ac.kr}

\date{\today}



\begin{document}
 
\maketitle 

\begin{abstract} 
The one-dimensional Euler-Poisson system arises in the study of phenomena of plasma such as plasma solitons, plasma sheaths, and double layers. When the system is rescaled by the Gardner-Morikawa transformation, the rescaled system is known to be formally approximated by the Korteweg-de Vries (KdV) equation. 
In light of this, we show existence of solitary wave solutions of the Euler-Poisson system in the stretched moving frame given by the transformation, and prove that they converge to the solitary wave solution of the associated KdV equation as the small amplitude parameter tends to zero. 
Our results assert that the formal expansion for the rescaled system is mathematically valid in the presence of solitary waves and justify Sagdeev's formal approximation for the solitary wave solutions of the pressureless Euler-Poisson system. Our work extends to the isothermal case.\\

\end{abstract}

\section{Introduction}
We consider the Euler-Poisson system on $\mathbb{R}$
\begin{equation}\tag{EP}\label{EP2}
\left\{
\begin{array}{l l}
\partial_{t} n + \partial_{x}(nu) = 0, \\ 
\partial_{t} u  + u\partial_{x} u + \sigma\frac{\partial_{x} n}{n} = -\partial_{x}\phi, \\
\partial_{x}^2\phi = e^\phi - n,
\end{array} 
\right.
\end{equation}
with the far-field condition
\begin{equation}\tag{BC}\label{bdCon x}
n \to 1, \quad u \to 0, \quad \phi \to 0 \quad \text{as} \quad |x| \to \infty,
\end{equation}
where $\sigma \geq 0$ is a constant of temperature. The system  \eqref{EP2} is referred to as the  \textit{pressureless} Euler-Poisson system when $\sigma=0$, and the \textit{isothermal} Euler-Poisson system when $\sigma >0$, respectively.  

We prove that for some $V,\gamma>0$ and sufficiently small $\veps>0$, \eqref{EP2}--\eqref{bdCon x} admits a unique (up to shift)  smooth traveling wave solution of the form $(n,u,\phi)(\xi)$, where
\begin{equation}\label{GMF}
\xi :=\veps^{1/2}\left( x-(V+\gamma\veps) t \right),
\end{equation}
and that the solutions converge to the solitary wave solution of the associated Korteweg–de Vries (KdV) equation  as $\veps \to 0$.
More specifically, for $V=\sqrt{1+\sigma}$,  \textit{the speed of ion sound},
 there exist positive constants $\veps_1$, $C$ and $\alpha$ such that for all $\veps\in(0,\veps_1)$, there holds 
\begin{equation}\label{Main Result}
\sup_{\xi\in\mathbb{R}}\left[ \left( |n-1-\veps n_{\text{KdV}}| + |u-\veps V n_{\text{KdV}}| + |\phi - \veps n_{\text{KdV}}| \right) e^{\alpha|\xi|/2} \right] \leq  C \veps^2,
\end{equation}
where $C$ and $\alpha$ are independent of $\veps$ and $\xi$, and
\begin{equation}\label{solutionKdV} 
n_{\text{KdV}}(\xi) := \frac{3\gamma}{V}\textnormal{sech}^2\left(\sqrt{\frac{V\gamma}{2}}\xi \right).
\end{equation}
Here we note that \eqref{solutionKdV} is the solitary traveling wave solution to the KdV equation
\begin{equation}\tag{KdV}\label{KdV}
\partial_{\bar{t}} v + V v\partial_{\bar{x}} v + \frac{1}{2V}\partial_{\bar{x}}^3 v = 0
\end{equation}
with a speed $\gamma>0$, that is,  $\xi = \bar{x}- \gamma \bar{t}$. In fact, we have that $\xi = \bar{x}- \gamma \bar{t}$ from \eqref{GMF}
and a linear transformation with specific scaling, the so-called Gardner-Morikawa transformation, 
\begin{equation}\tag{GM}\label{GM}
\bar{x} = \veps^{1/2}(x-Vt), \quad \bar{t} = \veps^{3/2}t.
\end{equation}
We refer to \cite{Gar} for more discussion of \eqref{GM}.


\subsection{Physicality of the model}
The motion of ions in plasma is governed by the  Euler-Poisson system
whose one-dimensional model\footnote{What mainly assumed in \eqref{EP1} are that $T_i$ and $T_e$ are constant (isothermal), the time variation of the magnetic field is negligible (electrostatic), and the dynamics of ions and electrons occur only in $X$ direction (plane wave). Additionally, it is assumed that $m_e=0$, where $m_e$ is the mass of the electrons, since $m_e \ll  m_i$ for every plasma environment.}
is given by:
\begin{equation}\label{EP1}
\left\{
\begin{array}{l l}
\partial_T n_i + \partial_X(n_iu_i) = 0, \\ 
\partial_T u_i  + u_i\partial_X u_i + \frac{k_BT_i}{m_i}\frac{\partial_X n_i}{n_i} = -\frac{e}{m_i}\partial_X\Phi, \\
\partial_X^2 \Phi = 4\pi e \left[ n_{e0}\exp\left(\frac{e\Phi}{k_BT_e}\right) - n_i \right],
\end{array} 
\right.
\end{equation}
where $n_i$, $u_i$ and $\Phi$ are unknown functions representing the density, velocity of the ion and the electric potential, respectively, and $m_i$, $T_i$, $k_B$ and $e$ represent the mass, temperature of the ion, the Boltzmann constant and the charge of an electron, respectively. Here we note that the model adopts the so-called Boltzmann relation stating that the distribution of electron density  is given by $n_e = n_{e0}\exp\left(\frac{e\Phi}{k_BT_e}\right)$, where $n_{e0}$ and $T_e$ are the electron density of the constant state and temperature of the electrons,  respectively. 

Upon an appropriate non-dimensionalization\footnote{$x= \frac{X}{\sqrt{k_BT_e/4\pi n_{e0}e^2}}, \;\; t = T\sqrt{ \frac{4\pi n_{e0}e^2}{m_i}}, \;\; n = \frac{n_i}{n_{e0}}, \;\; u = \frac{u_i}{\sqrt{k_BT_e/m_i}}, \;\; \phi = \frac{e\Phi}{k_BT_e}.$}, one can obtain \eqref{EP2} from \eqref{EP1}, where $\sigma = T_i/T_e\geq 0$  is a constant of the temperature ratio\footnote{In plasma physics, when one considers the fluid equations for plasma, it is often assumed that $\sigma=0$, which is an ideal situation for the  case $T_i \ll T_e$. In the case $T_i \approx T_e$, it is known that the fluid description is not relevant since the Landau damping effect becomes important.}.
We remark that the Euler-Poisson system \eqref{EP1} is a classical model of plasma at the fluid level, that is commonly used to study the ion dynamics, and is known to well describe several  interesting  phenomena of plasma  such as  the formation of plasma sheaths and double layers. We refer readers to \cite{Ch,Dav} for more physicality of \eqref{EP1}.

\subsection{Main results}
Before we discuss our main results, we present some preliminaries. 
In order to find solitary traveling wave solutions, we plug the Ansatz $(n,u,\phi)(\xi)$ into \eqref{EP2}--\eqref{bdCon x}, where $\xi$ is given by \eqref{GMF}.
Then one obtains the traveling wave equations,
\begin{subequations}\label{TravelEq}
\begin{align}[left = \empheqlbrace\,]
& -(V+\gamma\veps) n' + (nu)'= 0, \label{TravelEq1} \\
& -(V+\gamma\veps)u' + uu' +\sigma \frac{n'}{n}= -\phi',\label{TravelEq2}\\
& \veps\phi'' = e^\phi-n,\label{TravelEq3}
\end{align}
\end{subequations}
where $'$ denotes the derivative in $\xi$, with the far-field condition
\begin{equation}\label{bdCon+-inf}
n \to 1, \;\; u \to 0, \;\; \phi \to 0 \;\; \text{as} \;\; | \xi | \to  \infty.
\end{equation}
Thanks to translation invariance of the equation, $(n,u,\phi)(\xi-\xi_c)$ is also a solution to \eqref{TravelEq}--\eqref{bdCon+-inf} for any $\xi_c\in\mathbb{R}$. 
We shall see that the solution to \eqref{TravelEq}--\eqref{bdCon+-inf} is unique up to a translation. In light of this, we may assume without loss of generality  that $\xi_c=0$ in what follows. 

 We seek a non-trivial smooth solution $(n,u,\phi)$ to \eqref{TravelEq}--\eqref{bdCon+-inf} satisfying 
 the properties of \emph{solitary waves}.
\begin{defi*}\label{solitary-def}
 $(n,u,\phi)$ is called the \emph{solitary wave} solution to  \eqref{TravelEq}--\eqref{bdCon+-inf}
  if the following hold:
 \begin{enumerate}[(i)]
\item (symmetry)
\begin{equation}  
 n(\xi)=n(-\xi),\;\; u(\xi)=u(-\xi), \;\; \phi(\xi)=\phi(-\xi) \;\;  \text{for} \; \xi \in \mathbb{R}, \label{Symmetric} 
\end{equation}
\item (monotonicity)
\begin{equation}\label{mono}
 n'(\xi), \;\; u'(\xi), \;\; \phi'(\xi)<0  \quad  \text{for} \; \xi \in (0,\infty).
\end{equation}
\end{enumerate}
\end{defi*} 
Note that the solutions $(n,u,\phi)$ satisfying \eqref{bdCon+-inf}--\eqref{mono} achieve their unique maximum values $(n_*, u_*, \phi_*)$ at $\xi=0$, that is, 
\begin{equation}\label{max-v}
(n,u,\phi)(0) = (n_*, u_*, \phi_*),
\end{equation} and there hold
\begin{equation}\label{SignOfSols} 
n(\xi)>1, \;\; u(\xi)>0, \;\; \phi(\xi)>0 \quad \text{for} \;\; \xi\in \mathbb{R}. 
\end{equation}

We define some parameters.
For the case $\sigma>0$, one can check by inspection (See Appendix) that the equation
\begin{equation}\label{Aux3 lem-1}
z^{\sigma}\left[\sigma(z - 1)^2 + 1 \right]  =
\exp \left(\frac{\sigma}{2}\left(z^2 - 1 \right) \right)
\end{equation} 
has exactly two roots $z=1$ and $z=: \zeta_\sigma$ on $[1,\infty)$ such that 
$\zeta_{\sigma}>\sqrt{\frac{1+\sigma}{\sigma}}>1$ and
\begin{equation}\label{Aux4-1 lem}
z^{\sigma}\left[\sigma(z - 1)^2 + 1 \right]  >
\exp \left(\frac{\sigma}{2}\left(z^2 - 1 \right) \right) \;\; \text{for}\; z\in(1,\zeta_\sigma).
\end{equation}
Here we note that $\zeta_{\sigma}$  depends only on $\sigma$.
Let $\zeta_0$ be the unique positive root of
\begin{equation}\label{Eq z0 Cold-1}
z^2+1 = \exp(z^2/2).
\end{equation}
It is easy to check that $\zeta_0>1$ and 
\begin{equation}\label{Eq z_0 Cold}
z^2+1 > \exp(z^2/2) \;\; \text{for} \; z\in(0,\zeta_0).
\end{equation}
 
Let $(V, \gamma, \veps)$ be the positive numbers satisfying
\begin{subequations}\label{Cond_J}
\begin{align}[left = \empheqlbrace\qquad]
& \sqrt{\frac{1+\sigma}{\sigma}} < \frac{V+\gamma\veps}{\sqrt{\sigma}} < \zeta_{\sigma} & \text{when} \;\; \sigma>0, \label{Cond_JHot} \\
& 1 < V+\gamma\veps < \zeta_0 & \text{when} \;\; \sigma=0. \label{Cond_JCold}
\end{align}
\end{subequations}
Now we present our result on existence of  solitary waves. 
\begin{theorem}\label{MainThm} Suppose that $(V, \gamma, \veps)$ satisfies \eqref{Cond_J}.
Then the problem \eqref{TravelEq}--\eqref{bdCon+-inf} admits a unique (up to a shift) non-trivial smooth solution $(n,u,\phi)$. Moreover, upon an appropriate choice of shift, it satisfies \eqref{Symmetric}--\eqref{SignOfSols}. 
\end{theorem}

We remark that one can show, by inspection, that \eqref{Cond_J} is also a necessary condition for existence of non-trivial smooth solutions to \eqref{TravelEq}--\eqref{bdCon+-inf}. We omit the details since it is not of our main interest in the present work.

Although Theorem~\ref{MainThm}  
 holds true as long as \eqref{Cond_J} is satisfied, we restrict to the case $V=\sqrt{\sigma +1}$ and manipulate  only $\veps>0$ as a small parameter for fixed $V$ and $\gamma>0$ for Theorem~\ref{MainThm3}.
%
%
  
For the solitary wave solution $(n^\veps,u^\veps,\phi^\veps)$ to \eqref{TravelEq}--\eqref{bdCon+-inf} satisfying \eqref{Symmetric}--\eqref{SignOfSols}, we define the remainders as
\begin{equation}\label{Def nR uR phiR}
n_{R}^\veps:= n^\veps - 1 - \veps n_{\text{KdV}},\quad u_{R}^\veps := u^\veps - \veps Vn_{\text{KdV}},\quad \phi_{R}^\veps := \phi^\veps - \veps n_{\text{KdV}}.
\end{equation}
Let $f^{(k)}(z)$ denote the $k$-th order derivative of a function $f(z)$ in $z\in\mathbb{R}$.
Now we state our main theorem of the asymptotic behavior in the small amplitude limit. 
\begin{theorem}\label{MainThm3}
Let  $V=\sqrt{1+\sigma}$,
 $\sigma \geq 0$ and $\gamma>0$ be fixed. Let $k$ be any non-negative integer. Then there exist positive constants $\veps_1$, $\alpha$ (independent of $k$), and $C_{k}>0$ such that 
\begin{equation}\label{pointestimate}
  \sup_{\xi\in \mathbb{R}} \left|e^{\alpha|\xi|/2}   ({n_R^\veps}^{(k)}, {u_R^\veps}^{(k)} , {\phi_R^\veps}^{(k)}) (\xi) \right| \leq C_{k}\veps^2 
\end{equation}
for all $0<\veps<\veps_1$. Here $C_{k}$ and $\alpha$ are independent of $\veps$.
\end{theorem}

To prove Theorem \ref{MainThm}, 
we first derive a first-order ODE system for $(n,E)$ which is equivalent to \eqref{TravelEq}-\eqref{bdCon+-inf}, where   $E:=-\phi'$. Then we carry out a phase plane analysis in a similar fashion as  \cite{Cor}.
For the proof of Theorem \ref{MainThm3},
we first derive the equation for $\phi_R^\veps$:
\begin{equation}\label{IntroMainEq}
\left\{
\begin{array}{l l}
{\phi_{R}^\veps}'' - F_{\veps}\phi_R^\veps = \mathcal{M}_3^\veps,  \\ 
F_\veps(\xi) =  2V\gamma  - 2V^2 n_{\text{KdV}} - V^2\frac{\phi_R^\veps}{\veps},
\end{array} 
\right.
\end{equation}
where $\mathcal{M}_3^\veps$ is a function of $n_{\text{KdV}}, n^\veps_{R}, u^\veps_{R}$ and $\phi^\veps_{R}$. The choice of $V=\sqrt{1+\sigma}$ and the fact that $n_{\text{KdV}}$ satisfies \eqref{KdV} are crucially used to derive \eqref{IntroMainEq}.  One of the main difficulties for the analysis of the remainder equation \eqref{IntroMainEq} comes from the fact that $F_\veps(\xi)$ has different signs around $\xi=0$ and $\xi=\infty$. 
 By a careful analysis, we  obtain a  \emph{sharp} estimate for  the peak values of the solitary wave solution as follows.  For our notational convenience, we set the peak values by $(n_*^\veps, u_*^\veps, \phi_*^\veps) := (n^\veps,u^\veps,\phi^\veps)(0)$.  
\begin{proposition}\label{cor2}
Let $V=\sqrt{1+\sigma}$,  $\sigma \geq 0$ and $\gamma>0$ be fixed.
Then there exist positive constants $\veps_0$ and $C$ such that for all $0 < \veps< \veps_0$, 
\begin{equation}\label{unibd_nuphi}
\left| n_\ast^\veps - 1 - 3\gamma V^{-1}\veps \right| + \left|u^\veps_\ast - 3\gamma\veps \right| + \left| \phi^\veps_\ast - 3\gamma V^{-1}\veps \right| \leq \veps^2C,
\end{equation}
where $C$ is independent of $\veps$.
Moreover, $V= \sqrt{1+\sigma}$ is necessary for $\lim_{\veps\to 0}n_\ast^\veps = 1$.
\end{proposition}

 Thanks to Proposition~\ref{cor2}, we see 
that  $\lim_{\veps \to 0}F_\veps(0)=-4V\gamma<0$ by \eqref{unibd_nuphi} while $\lim_{\xi \to \infty}F_\veps(\xi)=2V\gamma$ for all $\veps$ by \eqref{bdCon+-inf}. Hence we  split the interval $[0,\infty)$ into two regions and carry out our analysis separately. 

On the other hand, we note that the coefficient $3\gamma/V$ of $\veps$ in \eqref{unibd_nuphi} is exactly the peak value of $n_{\text{KdV}}(\xi)$ defined in \eqref{solutionKdV}. That is, the order of $(n_R^\veps,u_R^\veps,\phi_R^\veps)(\xi)$ is at least $O(\veps^2)$  at $\xi=0$. This is crucially used in the proof of our main result. By Gronwall's inequality together with \eqref{unibd_nuphi}, one can obtain \textit{the local uniform estimate} for $\veps^{-2}\phi_R^\veps$ around $\xi=0$ (Proposition~\ref{Prop2}).

Now it remains to obtain \textit{the local uniform decay estimate} for $\veps^{-2}\phi_R^\veps$ around $\xi=\infty$,
for which  a careful analysis is required. It is not clear that if one can choose a uniform constant $\xi_1>0$ so that $F_\veps(\xi)$ is bounded from below by a positive constant on $[\xi_1,\infty)$ for all $\veps$; the time $\xi_\veps>0$ at which $F_\veps(\xi_\veps)=0$ is realized may tend to $\infty$ as $\veps \to 0$. Verifying existence of such $\xi_1$ is a key step toward the uniform estimate for $\veps^{-2}\phi_R^\veps$.
To this end, we first estimate the uniform lower bounds for the speed of trajectory curves $\veps^{-1}(n^\veps,E^\veps)(\xi)$ (Lemma~\ref{LemmaAux1}--\ref{LemmaAux3}). This is the place where the sharp estimate \eqref{unibd_nuphi} is crucially used again. Then we obtain the uniform decay estimate for $\veps^{-1}(n^\veps-1,u^\veps,\phi^\veps)(\xi)$ (Proposition~\ref{Prop1}), which yields \textit{the local uniform decay estimate} for $\veps^{-2}\phi_R^\veps$ around $\xi=\infty$ (Proposition~\ref{Prop3}). Combining Proposition \ref{Prop2}--\ref{Prop3} and using the fact that $\phi_R^\veps$ is symmetric about $\xi=0$, we obtain \textit{the uniform decay estimate} for $\veps^{-2}\phi_R^\veps$ on $\mathbb{R}$. 
The estimates for $\veps^{-2}n_R^\veps$ and $\veps^{-2}u_R^\veps$  immediately follow from \eqref{Eq for difference}.



\subsection{Motivation and related results}

Among others, the propagation of solitary waves is one of the most interesting phenomena in the dynamics of electrostatic plasma. Motivated by the discovery that the KdV equation, originally derived to describe the motion of water waves, can be also derived in different contexts such as the study of hydromagnetic waves in plasma \cite{Gar}, plasma physicists have found formal connections \cite{Sag,Wa} between the pressureless Euler-Poisson system and the KdV equation. Later on, the phenomenon was experimentally observed in \cite{ikezi}.
We shall briefly illustrate how the equations \eqref{EP2} and \eqref{KdV} are related, and discuss some relevant mathematical results.



By introducing \eqref{GM}, \eqref{EP2} becomes
\begin{equation}\tag{GMEP}\label{EPGM}
\left\{
\begin{array}{l l}
\veps\partial_{\bar{t}} n -V\partial_{\bar{x}}n +  \partial_{\bar{x}}(nu) = 0, \\ 
\veps\partial_{\bar{t}} u  - V\partial_{\bar{x}} u + u\partial_{\bar{x}} u + \sigma\frac{\partial_{\bar{x}} n}{n} = -\partial_{\bar{x}}\phi, \\
\veps\partial_{\bar{x}}^2\phi = e^\phi - n.
\end{array} 
\right.
\end{equation}
It is found at the formal level in \cite{Wa} for the case of $\sigma=0$ and $V=1$,  that if we assume that the solution $(n,u,\phi)$ to \eqref{EPGM} is given by a formal expansion 
\begin{equation}\label{FormalExpan}
n = 1 + \sum_{k=1}^\infty \veps^k n_{(k)}, \quad  
u = \sum_{k=1}^\infty \veps^k u_{(k)}, \quad 
\phi = \sum_{k=1}^\infty \veps^k \phi_{(k)},
\end{equation}
then each component of $(n_{(1)},u_{(1)},\phi_{(1)})$ satisfies \eqref{KdV}. A mathematical validity of this result has not been proved until the work of \cite{Guo}, in which    
  it is proved that the solutions to \eqref{EPGM} with the \emph{prepared} initial data converge as $\veps \to 0$ to that of \eqref{KdV} on any fixed time interval.
  This result deals also with the case $\sigma > 0$ with the choice of $V=\sqrt{1+\sigma}$. In this case, each component of $(n_{(1)}, V^{-1} u_{(1)} ,\phi_{(1)})$ satisfies \eqref{KdV}.
  
On the other hand, a \emph{formal} approximation of the solitary wave solution to \eqref{EP2} with $\sigma=0$ in terms of the hyperbolic secant function are obtained in \cite{Sag}. Suppose that $(n,u,\phi)(\bar{\xi})$, where $\bar{\xi} = x - Mt$ for a constant $M>0$, is a solitary wave solution to \eqref{EP2}. Imposing $n-1, u, \phi \to 0$ as $\xi \to -\infty$, 
\eqref{EP2} is reduced to the equation
\begin{equation}\label{Sag3}
\phi'' = e^\phi - \frac{M}{\sqrt{M^2-2\phi}}.
\end{equation}
It is shown that $1<M<\zeta_0$ is a necessary (\cite{Sag}) and sufficient (\cite{Satt}) condition for existence of the solitary wave solution to \eqref{Sag3}, where $\zeta_0$ is a unique positive root of \eqref{Eq z0 Cold-1}. Let $M=1+\gamma\veps$ for $0<\gamma\veps \ll 1$ and assume $0<\phi\ll 1$. Expanding the right-hand side (RHS) of \eqref{Sag3} for $\gamma\veps$ and $\phi$, one can  obtain
\[
\phi''-2\gamma\veps\phi+\phi^2=O\left(|\phi|(|\gamma\veps|^2 + |\phi|^2) \right).
\]
Neglecting the higher order terms, one can formally obtain 
\begin{equation}\label{Sag6}
\phi \approx 3\gamma\veps \, \text{sech}^2\left(\sqrt{2^{-1}\gamma\veps}\,\bar{\xi}\right)=3\gamma\veps \, \text{sech}^2\left(\sqrt{2^{-1}\gamma\veps}\,[x-(1+\gamma\veps)t]\right).
\end{equation}
By introducing \eqref{GMF}, we observe that the right-hand side of \eqref{Sag6} is the same as $\veps n_{\text{KdV}}$ when $\sigma=0$, our rigorous approximation of $\phi$. 

We remark that the work of \cite{Guo} well describes the asymptotic behavior ($\veps\to0$) of the small solutions to \eqref{EP2} up to the time scale of $t=O(\veps^{-3/2})$. 
However, this setting is not yet appropriate to illuminate  the phenomena of solitary waves in plasma.
Also, the formal approximation \eqref{Sag6} of \cite{Sag} has not yet been completely justified from a mathematical point of view. On the other hand, the work of \cite{Sche} for the pressureless model 
 uses the formal approximation \eqref{Sag6} in their analysis of the linear stability of the solitary wave solution to \eqref{EP2}. However, we are not aware of any literature proving validity of the approximation in a rigorous manner.

%
Our results  assert that    the  formal expansion \eqref{FormalExpan} for \eqref{EPGM} as in \cite{Wa} is mathematically valid in the presence of solitary waves, and also justify the formal computation of \eqref{Sag6} in \cite{Sag}. Moreover our results extend 
 to the isothermal Euler-Poisson system, that is, \eqref{EP2} with $\sigma > 0$.

The isothermal system \eqref{EP2}, unlike the pressureless case, can not be reduced to the explicit second-order ODE for $\phi$. Instead of dealing with the implicit ODE, we consider the explicit ODE system \eqref{ODE_n_E} for $(n,-\phi')$.

We  present some numerical experiments in Figure \ref{FigNumeric}, which demonstrate that $\veps^{-1}(n^\veps-1)$ converges to $n_{\text{KdV}}$ as $\veps$ tends to zero. 
%
We employ the Runge-Kutta method of order 4 to solve the initial value problem for the ODE system \eqref{ODE_n_E} where the  initial values are suitably chosen by using the first integral \eqref{1st Int}--\eqref{Def_g} satisfying the far-field condition \eqref{bdCon n E}.

The paper is organized as follows. 
Section 2--3 are devoted to the proof of Theorem \ref{MainThm} and Proposition \ref{cor2}, respectively. In Section 4.1, we obtain the uniform decay estimates for $\veps^{-1}(n^\veps-1,u^\veps,\phi^\veps)(\xi)$ and their higher order derivatives. In Section 4.2, we derive the remainder equations and finish the proof of Theorem \ref{MainThm3}. Section 5 contains miscellaneous calculations. For notational simplicity, we set
\begin{equation}\label{J}
J = J(\veps):= V+\gamma\veps = \sqrt{1+\sigma} + \gamma\veps
\end{equation}
for fixed $V=\sqrt{1+\sigma}$ and $\gamma>0$ throughout Section 3--5. $C$ denotes a generic constant independent of $\veps$ and $\xi$.
\begin{figure}[h]
\begin{tabular}{cc}
\resizebox{60mm}{!}{\includegraphics{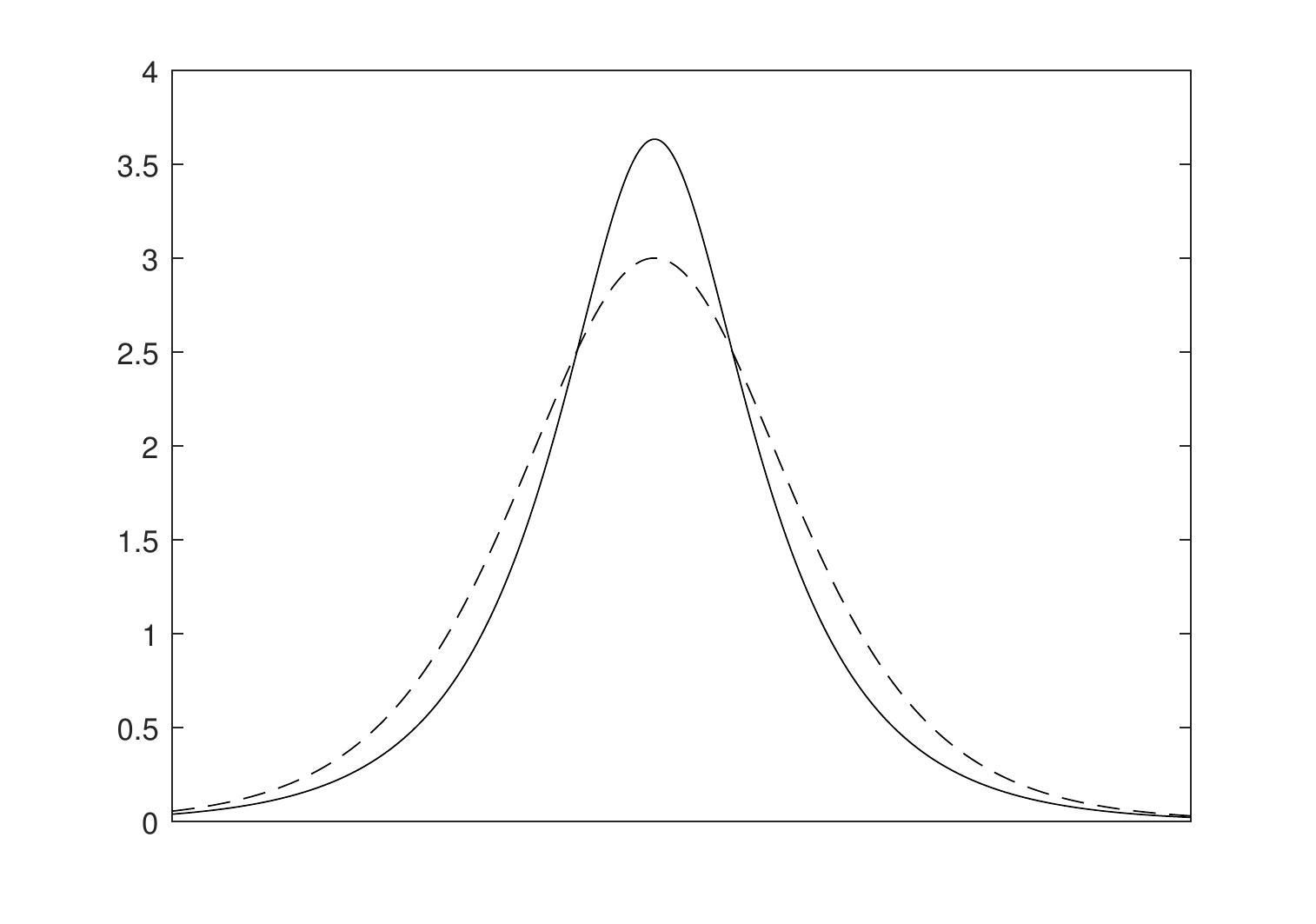}}  & \resizebox{60mm}{!}{\includegraphics{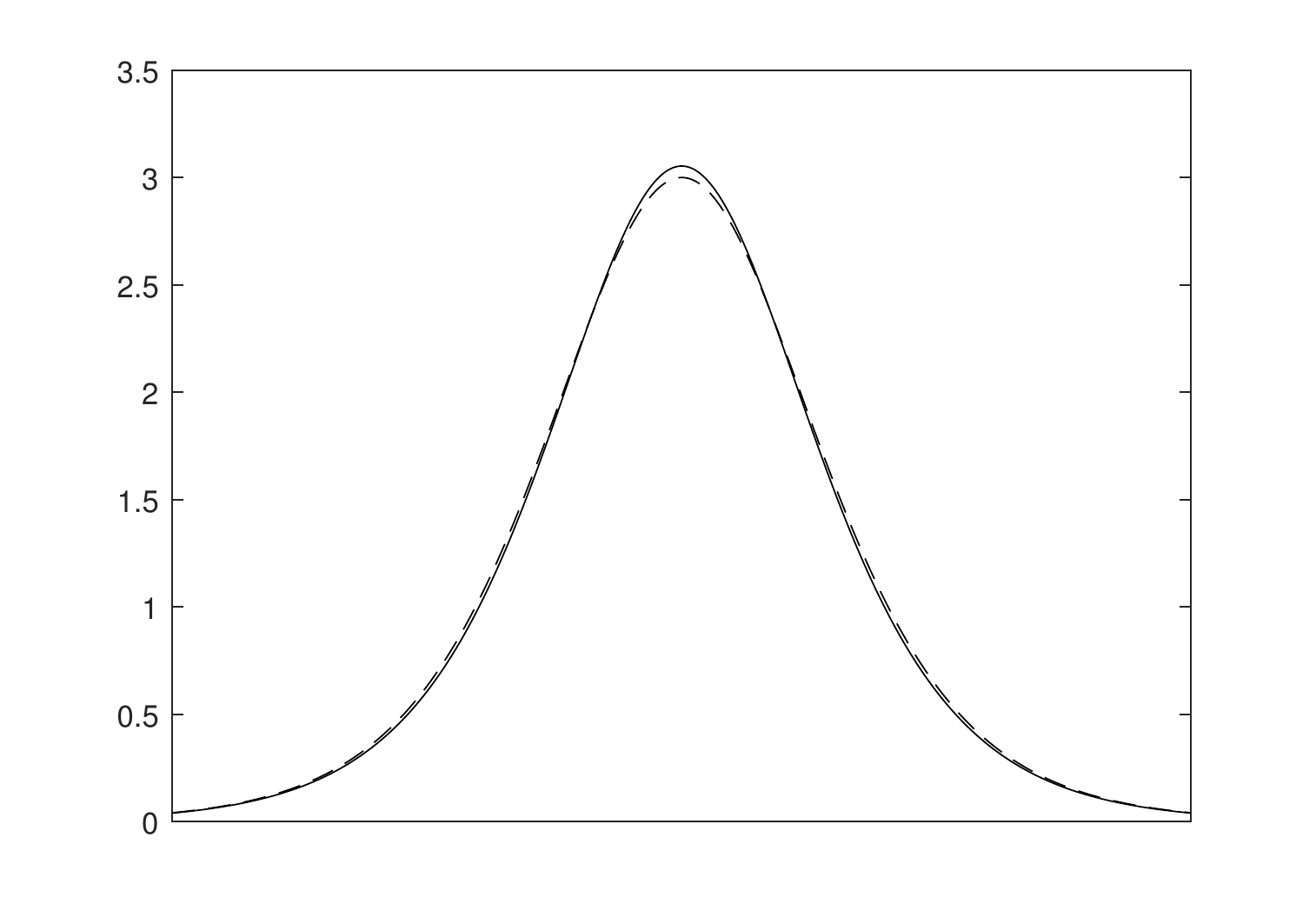}} \\
(a) $\sigma=0$ and $\veps=0.1$ & (b) $\sigma=0$ and $\veps=0.01$\\
\resizebox{60mm}{!}{\includegraphics{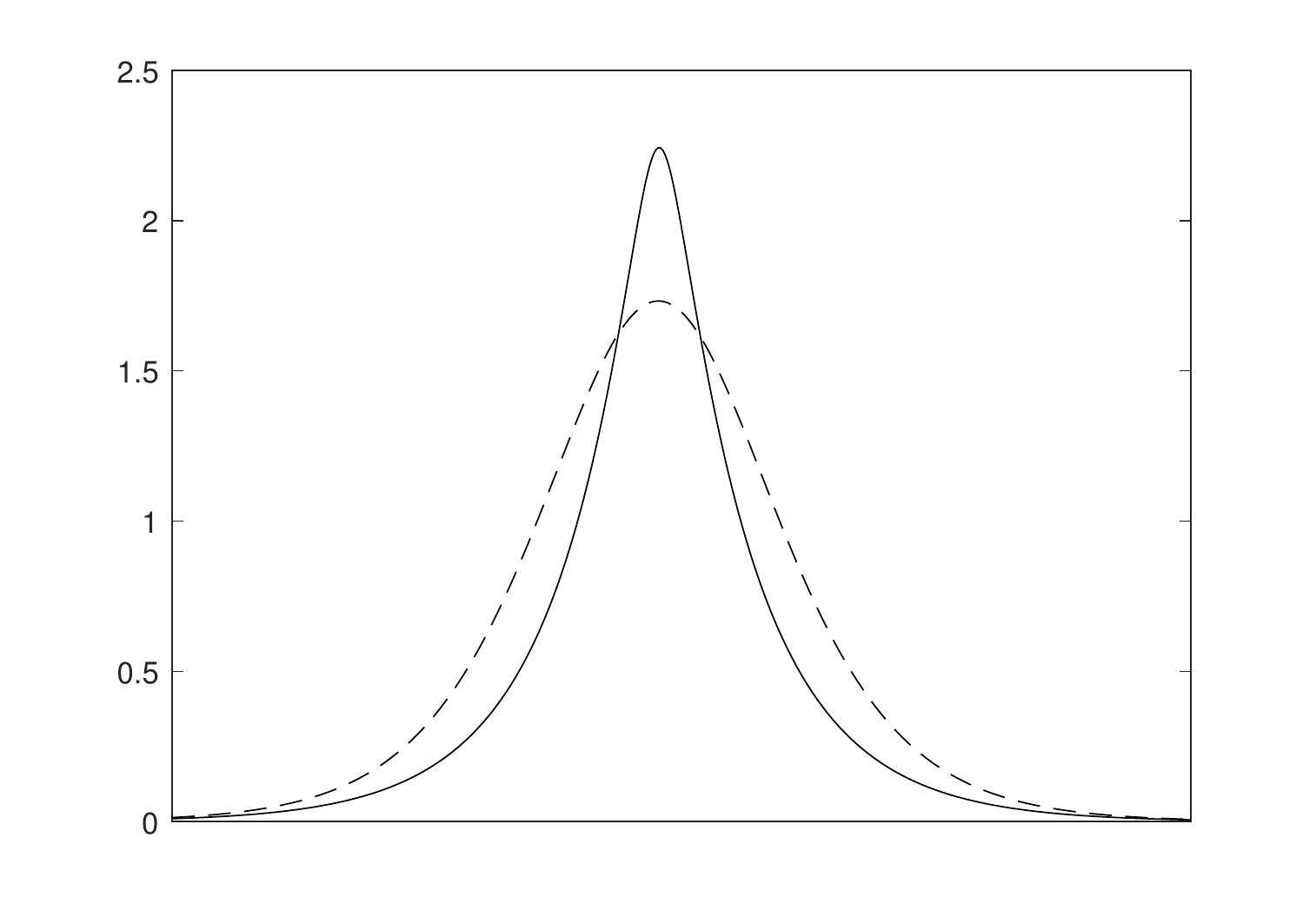}} & \resizebox{60mm}{!}{\includegraphics{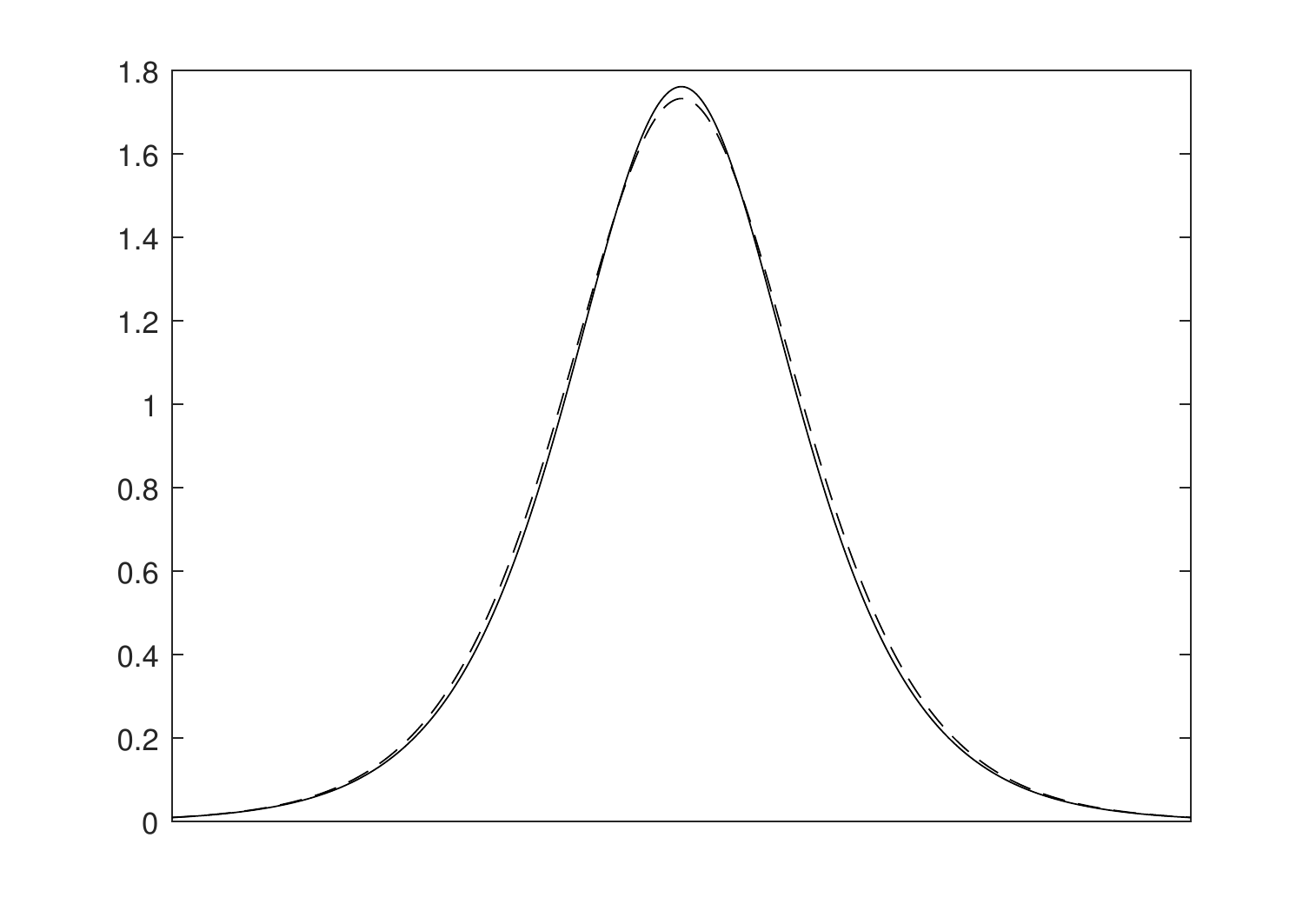}} \\
(c) $\sigma=2$ and $\veps=0.1$ & (d) $\sigma=2$ and $\veps=0.01$
\end{tabular}
\caption{Numerical experiments for comparison of $\veps^{-1}(n^\veps-1)$ (solid) with $n_{\text{KdV}}$ (dashed) in the stretched moving frame $\xi=\veps^{1/2}(x-(\sqrt{1+\sigma}+\veps t) )$.}
\label{FigNumeric}
\end{figure}

\section{Existence of solitary wave solution}
In this section, we carry out a phase plane analysis. To this end, we first reduce \eqref{TravelEq} to a first-order ODE system for $(n,-\phi')$.
\subsection{Reduction to a first-order ODE system}
Suppose that $(n,u,\phi)$ is a smooth solution to \eqref{TravelEq}--\eqref{bdCon+-inf}. Integrating \eqref{TravelEq1}--\eqref{TravelEq2} from $-\infty$ to $\xi$, we have
\begin{subequations}\label{TravelEqA}
\begin{align}[left = \empheqlbrace\,]
& -(V+\gamma\veps)n+nu = -(V+\gamma\veps), \label{TravelEqA1} \\ 
& -(V+\gamma\veps)u + \frac{1}{2}u^2 + \sigma\ln n = -\phi. \label{TravelEqA2}
\end{align}
\end{subequations}
By a simple manipulation, \eqref{TravelEqA} can be written as
\begin{subequations}\label{TravelEqB}
\begin{align}[left = \empheqlbrace\,]
& u=(V+\gamma\veps)\left( 1-\frac{1}{n} \right), \label{TravelEqB1} \\ 
& \phi = H(n), \label{TravelEqB2}
\end{align}
\end{subequations}
where 
\begin{equation}\label{Def_H}
H(n) : = \frac{(V+\gamma\veps)^2}{2}\left(1 - \frac{1}{n^2} \right)  - \sigma\ln n.
\end{equation}
Differentiating \eqref{TravelEqB2} in $\xi$, we have
\begin{equation}\label{Aux 2}
\phi' =  h(n)n',
\end{equation}
where
\begin{equation}\label{Def_h}
h(n):= \frac{dH(n)}{dn} = \frac{(V+\gamma\veps)^2}{n^3} - \frac{\sigma}{n}.
\end{equation}
By \textit{defining} a function 
\begin{equation}\label{Def_E}
E(\xi):=-\phi'(\xi),
\end{equation}
we obtain an \textit{ODE system for $n$ and $E$} from \eqref{Aux 2} and \eqref{TravelEq3}:
\begin{subequations}\label{ODE_n_E}
\begin{align}[left = \empheqlbrace\,]
& -h(n) n' = E, \label{ODE_n_E1} \\
& \veps  E' = n- e^{H(n)}.\label{ODE_n_E2}
\end{align}
\end{subequations}
Multiplying \eqref{ODE_n_E2} by $E=-\phi'$, and then using \eqref{ODE_n_E1} and \eqref{Def_h}, we get
\begin{equation}\label{Aux 4}
\begin{split}
\frac{\veps}{2}\left(E^2\right)'
& = -nh(n) n' + e^{H(n)}h(n) n' \\
& = \left( \frac{(V+\gamma\veps)^2}{n} + \sigma n + e^{H(n)}\right)'.
\end{split}
\end{equation}
Integrating \eqref{Aux 4} in $\xi$, we have 
\begin{equation}\label{Aux 5}
\frac{\veps}{2} E^2 =  \frac{(V+\gamma\veps)^2}{n} + \sigma n + e^{H(n)} + c
\end{equation}
for some constant $c$. From \eqref{Aux 5} and \eqref{bdCon+-inf}, we see that $n$ and $E$ must satisfy
\begin{equation}\label{bdCon n E}
n \to 1, \quad  E \to 0 \quad \text{as}\quad  |\xi| \to \infty.
\end{equation}

Conversely, suppose that $(n,E)$ is a smooth solution to \eqref{ODE_n_E} with \eqref{bdCon n E}. Defining $u$ and $\phi$ as \eqref{TravelEqB}, we see that $(n,u,\phi)$ satisfies \eqref{TravelEq1}-\eqref{TravelEq2} and \eqref{bdCon+-inf}. Moreover, from \eqref{Aux 2} and \eqref{ODE_n_E1}, we have $-\phi'=E$. Thus $(n,u,\phi)$ also satisfies \eqref{TravelEq3} by \eqref{ODE_n_E2}. We remark that \eqref{Aux 5} gives a first integral of the system \eqref{ODE_n_E} with \eqref{bdCon n E}:
\begin{equation}\label{1st Int}
\frac{\veps}{2}E^2 = g(n) -g(1),
\end{equation}
where
\begin{equation}\label{Def_g}
g(n):=  \frac{(V+\gamma\veps)^2}{n} + \sigma n + e^{H(n)}.
\end{equation}

\subsection{Stationary points}
We examine the stationary points of \eqref{ODE_n_E} on $(n,E)$ plane. From \eqref{Def_h}, we see that when $\sigma>0$, 
\begin{equation}\label{sign h}
h(n)
\left\{
\begin{array}{l l}
=0 \;\; \text{for} \;\; n=n_s, \\ 
<0 \;\; \text{for} \;\; n>n_s, \\
>0 \;\; \text{for} \;\; 0<n<n_s,
\end{array}
\right.
\end{equation}
where
\begin{equation}\label{Def n^s}
n_s:=\frac{V+\gamma\veps}{\sqrt{\sigma}},
\end{equation}
and when $\sigma=0$, 
\begin{equation}\label{sign h cold}
h(n)>0 \;\; \text{for} \;\; n>0.
\end{equation}
Henceforth, we will assume 
\begin{enumerate}
\item[\textbf{(A1)}] $n<n_s$ \quad when $\sigma>0$.
\end{enumerate}
We will see later that \textbf{(A1)} is valid for our analysis. Since $h>0$ for $\sigma \geq 0$ under the assumption \textbf{(A1)}, the stationary points of the system \eqref{ODE_n_E} are where 
\begin{equation}\label{Aux 8}
\frac{E}{h(n)} = 0 \quad \text{and} \quad  n=e^{H(n)}
\end{equation}
hold. Hence, all stationary points lie on $n$ axis. For $n>0$, we have
\begin{equation}\label{Aux 9}
n=e^{H(n)} \quad \Longleftrightarrow \quad  l(n):=\ln n - H(n)=0.
\end{equation}
One may easily check that the function $l(n)$ is strictly decreasing on $(0,n_c)$ and strictly increasing on $(n_c,\infty)$, where
\begin{equation}\label{Def n_c}
n_c: = \frac{V+\gamma\veps}{\sqrt{1+\sigma}}.
\end{equation}
The condition \eqref{Cond_J} implies that $1 < n_c$ for $\sigma\geq 0$. Moreover, since $l(1) = 0$ and $\lim_{n \to \infty}l(n) = \infty$, it follows that $l(n)$ has only two zeros $1$ and $n_{ce}$ such that
\begin{equation}\label{Aux 10}
1 <n_c< n_{ce}.
\end{equation}
Thus the system \eqref{ODE_n_E} has only two stationary points $(1,0)$ and $(n_{ce},0)$. For a later purpose, we observe that for $\sigma \geq 0$,
\begin{equation}\label{Aux 6}
\left\{
\begin{array}{l l}
n  <  e^{H(n)} \quad \text{for} \; n \in (1,n_{ce}), \\ 
n  >  e^{H(n)} \quad \text{for} \; n \in (n_{ce},\infty).
\end{array} 
\right.
\end{equation}

\subsection{Local behavior}
The Jacobian of the system \eqref{ODE_n_E} is
\begin{equation}\label{jacobian}
\left( 
\begin{array}{cc}
\frac{E}{[h(n)]^2}\frac{dh(n)}{dn} & \frac{-1}{h(n)} \\
\frac{1}{\veps}\left[1-h(n)e^{H(n)} \right] & 0
\end{array}
\right).
\end{equation}
The trace of \eqref{jacobian} is zero at all stationary points $(1,0)$ and $(n_{ce},0)$. Since $n=1$ is a solution to \eqref{Aux 9}, we obtain from \eqref{sign h}--\eqref{sign h cold} and the condition \eqref{Cond_J} that
\[
\begin{split}
\frac{1 - h(1)e^{H(1)}}{\epsilon h(1)}
& = \frac{1+\sigma -(V+\gamma\veps)^2}{\epsilon h(1)} < 0.
\end{split}
\]
Hence, $(1,0)$ is a saddle point for $\sigma \geq 0$. Since $n=n_{ce}$ is a solution to \eqref{Aux 9}, we have from \eqref{Def n_c}--\eqref{Aux 10} that
\begin{equation}\label{AuxCalcul1}
\begin{split}
1 - h(n_{ce})e^{H(n_{ce})}
& = 1 + \sigma - \frac{(V+\gamma\veps)^2}{(n_{ce})^2}  \\
&  = (1 + \sigma)\left(1 -\frac{(n_c)^2}{(n_{ce})^2} \right) > 0.
\end{split}
\end{equation}
From \eqref{sign h cold} and \eqref{AuxCalcul1}, we see that $(n_{ce},0)$ is a center 
in the case $\sigma=0$. On the other hand, in the case $\sigma>0$, $(n_{ce},0)$ can be a saddle or a center since the sign of $h$ changes depending on the position of $n_{ce}$ with respect to $n_s$ (See \eqref{sign h}). It will be shown later that the condition \eqref{Cond_JHot} implies that
\begin{enumerate}
\item[\textbf{(A2)}] $n_{ce} < n_s$ when $\sigma>0$.
\end{enumerate}
We assume \textbf{(A2)} so that $(n_{ce},0)$ is a center.

\subsection{Direction of vector fields}
From \eqref{Aux 6} and the ODE \eqref{ODE_n_E2}, we find that $E'<0$ in the region where $ 1 < n < n_{ce}$, and $E' >0$ in the region where $n> n_{ce}$. Since $h(n)>0$ by \eqref{sign h}, \eqref{sign h cold} and \textbf{(A1)}, we see from the ODE \eqref{ODE_n_E1} that $n'<0$ in the region where $E>0$, and $n'>0$ in the region where $E<0$ (See Figure \ref{FigTrajec}).

\begin{figure}[h]
    \centering
        \includegraphics[scale=0.3]{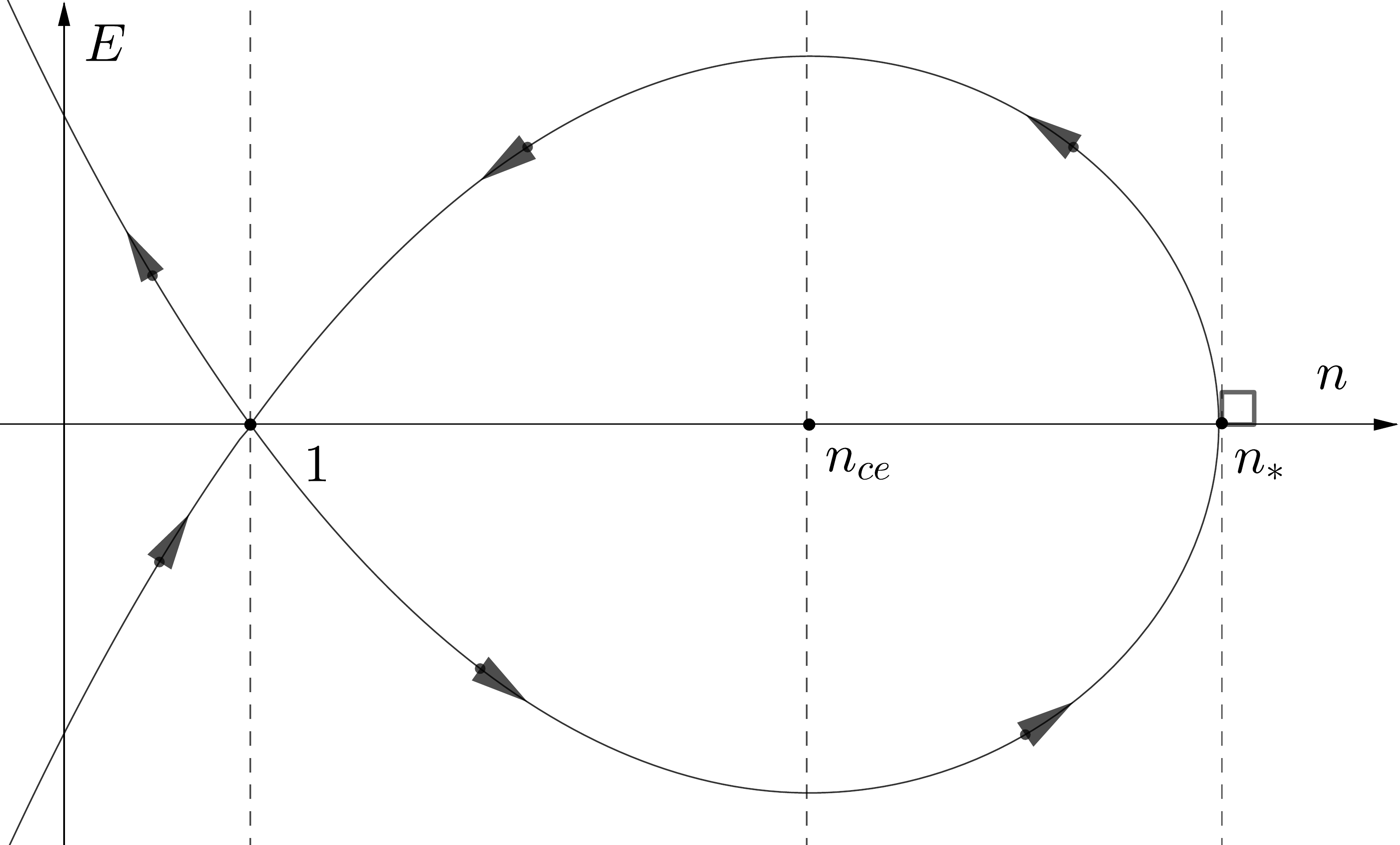}
        \caption{Trajectory curve}
        \label{FigTrajec}
\end{figure}

\subsection{First integral}
By the first integral \eqref{1st Int}--\eqref{Def_g}, the trajectory starting from the point $(1,0)$ with $E<0$ satisfies $E(n) = -\frac{\sqrt{2(g(n)-g(1))}}{\sqrt{\veps}}$. We have 
\begin{equation}\label{EPrime}
\frac{dE}{dn}(n) = -\frac{dg}{dn}(n)\frac{1}{\sqrt{2\veps(g(n)-g(1))}},
\end{equation}
where
\begin{equation}\label{gPrime}
\begin{split}
\frac{dg}{dn}(n)
& = -\frac{(V+\gamma\veps)^2}{n^2} + \sigma  +h(n) e^{H(n)} \\
& = -h(n)\left(n-  e^{H(n)}\right).
\end{split}
\end{equation}

We recall that $n=1$ and $n=n_{ce}$ satisfy \eqref{Aux 9} and that they are only such points. Hence, in the case $\sigma>0$, \eqref{gPrime} vanishes only at $1$, $n_{ce}$ and $n_s$. From \textbf{(A2)}, \eqref{sign h} and \eqref{Aux 6}, we see that $g(n)$ strictly increases on $(1,n_{ce})$ and strictly decreases on $(n_{ce}, n_s)$. This observation shows that if 
\begin{equation}\label{Cond g}
g(1)>g(n_s),
\end{equation}
then there exists a unique $n_\ast$ satisfying (See Figure \ref{FigGraph g})
\begin{equation}\label{g(n)=g(1)hot}
1<n_{ce} < n_\ast < n_s \quad  \text{and} \quad g(n_\ast)=g(1).
\end{equation}
It will be shown in Lemma \ref{LemCond_g} that the condition \eqref{Cond_JHot} implies \eqref{Cond g}. 

In the case $\sigma = 0$, $g(n)$ strictly decreases on $(n_{ce},\infty)$ since $h(n)>0$ for $n>0$. From  \eqref{Eq z_0 Cold} and the condition \eqref{Cond_JCold}, we have 
$$\lim_{n \to \infty}g(n) = \exp\left(\frac{(V+\gamma\veps)^2}{2}\right) < (V+\gamma\veps)^2+1 =g(1).$$
Therefore, there exists a unique $n_\ast$ such that 
\begin{equation}\label{g(n)=g(1)cold}
1<n_{ce} < n_\ast \quad \text{and} \quad g(n_\ast)=g(1).
\end{equation}
 
\begin{figure}[h]
        \begin{subfigure}[b]{0.5\textwidth}
        \includegraphics[scale=0.18]{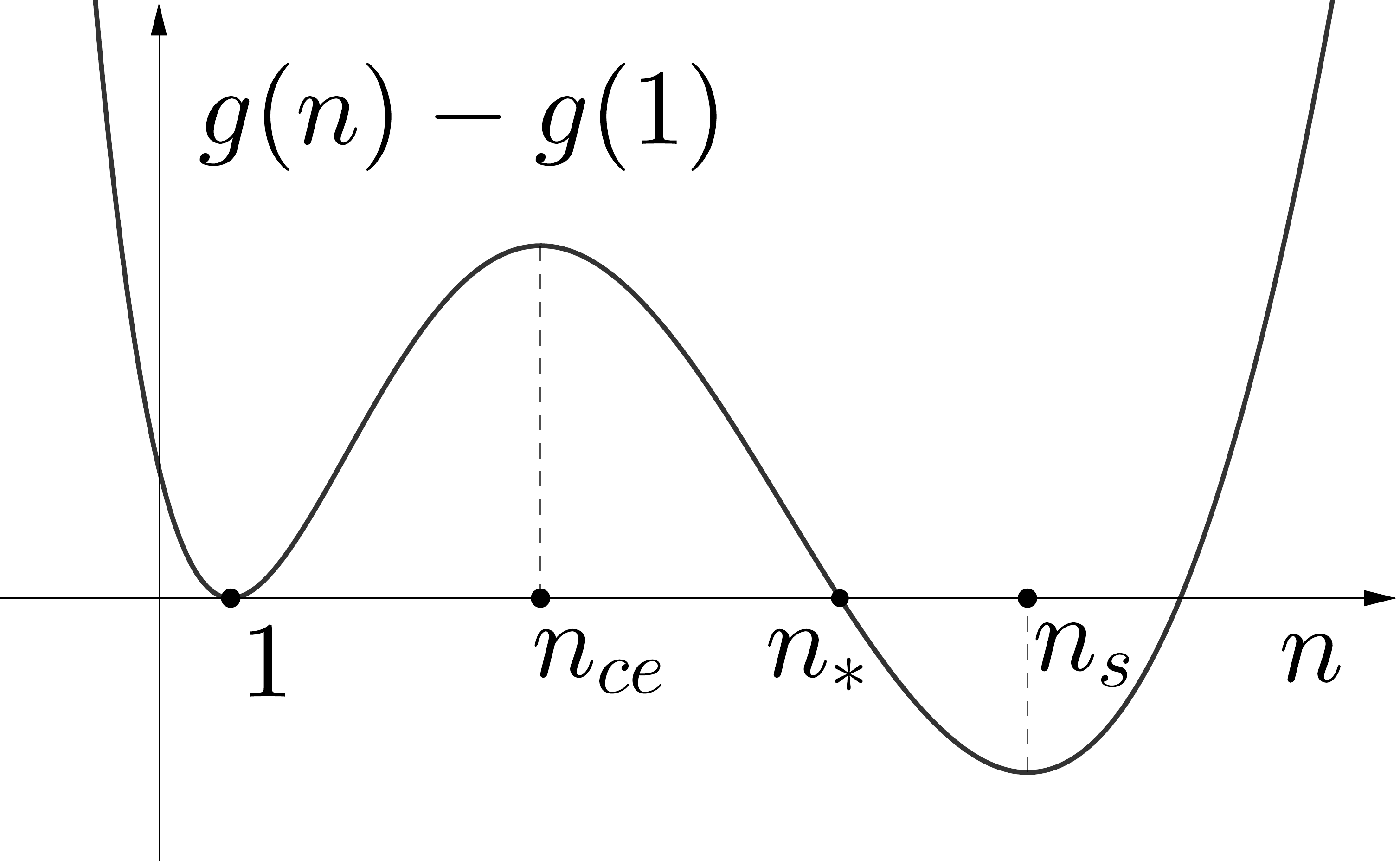}
        \caption{$\sigma>0$}
        \label{graphfofghot}
    \end{subfigure}
    ~  
    \begin{subfigure}[b]{0.5\textwidth}
        \includegraphics[scale=0.18]{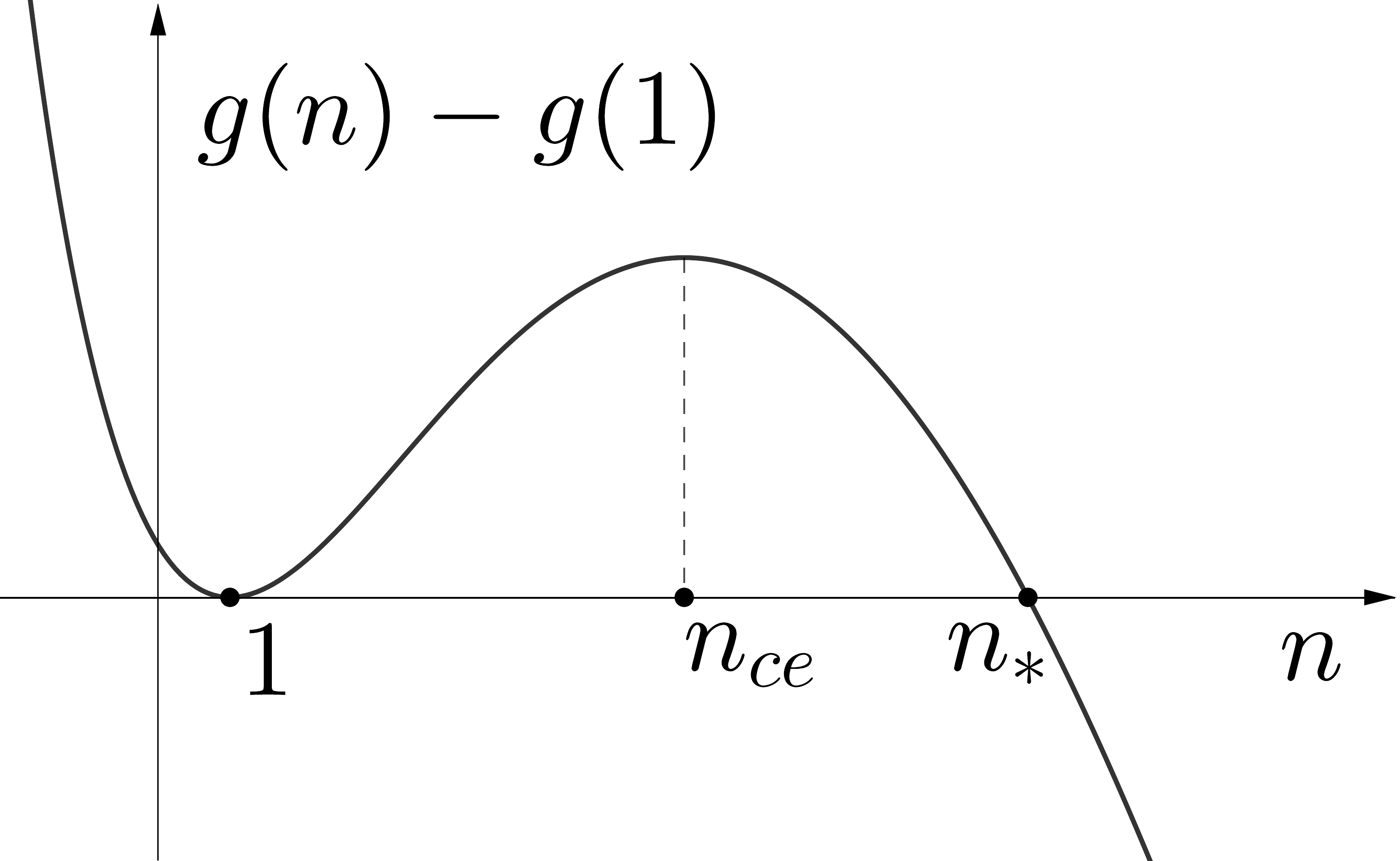}
        \caption{$\sigma=0$}
        \label{graphfofgcold}
    \end{subfigure}
    \caption{Graph of $g(n)-g(1)$}\label{graphofg}
    \label{FigGraph g}
\end{figure}

\begin{lemma}\label{LemCond_g}
The condition \eqref{Cond_JHot} implies that $g(1)>g(n_s)$.
\end{lemma}

\begin{proof}[Proof of Lemma \ref{LemCond_g}]
From the definition \eqref{Def_g} of $g$, we need to show that
\begin{equation}\label{Aux1 lem}
(V+\gamma\veps)^2 + \sigma + 1 > \frac{(V+\gamma\veps)^2}{n_s} + \sigma n^s + e^{H(n_s)}.
\end{equation}
By the definition \eqref{Def n^s} of $n_s$, \eqref{Aux1 lem} can be written as 
\begin{equation}\label{Aux2 lem}
\sigma(n_s)^2 + \sigma + 1 > 2\sigma n_s +  e^{H(n_s)}.
\end{equation}
On the other hand, using the definitions \eqref{Def_H} and \eqref{Def n^s} of $H$ and $n_s$, we have
\[
\begin{split}
H(n_s)
& =\frac{(V+\gamma\veps)^2}{2}\left(1 - \frac{1}{(n_s)^2} \right)  - \sigma\ln n_s  \\
& = \frac{\sigma}{2}\left((n_s)^2 - 1 \right) - \sigma\ln n_s.
\end{split}
\]
Therefore, \eqref{Aux2 lem} is equivalent to
\[
\sigma(n_s)^2 + \sigma + 1 > 2\sigma n_s + (n_s)^{-\sigma} \exp \left(\frac{\sigma}{2}\left((n_s)^2 - 1 \right) \right),
\]
and equivalently, we have
\begin{equation}\label{Aux4 lem}
(n_s)^{\sigma}\left[\sigma(n_s - 1)^2 + 1 \right] > \exp \left(\frac{\sigma}{2}\left((n_s)^2 - 1 \right) \right).
\end{equation}
From the definition of $n_s$, \eqref{Aux4-1 lem}, and \eqref{Aux4 lem}, we see that \eqref{Cond_JHot} implies \eqref{Cond g}.
\end{proof}

\begin{proof}[Proof of Theorem \ref{MainThm}]
We first observe that Lemma \ref{LemCond_g} implies \textbf{(A2)}. Suppose that $n_{ce} \geq n_s$. By applying  \eqref{sign h} and \eqref{Aux 6}, we see from \eqref{gPrime} that $g(n)$ strictly increases on $(1,n_s)$. This yields that $g(1) < g(n_s)$, which contradicts $g(1)>g(n_s)$.

From \eqref{EPrime}-\eqref{gPrime} and  \eqref{g(n)=g(1)hot}--\eqref{g(n)=g(1)cold}, we have for $\sigma \geq 0$ that
\[
\lim_{n \to n_\ast^-}\frac{dE}{dn}(n) = +\infty.
\]
Thus, by a standard argument, one can show that the trajectory starting from the saddle point $(1,0)$ with $E<0$ reaches at $(n_\ast,0)$ in a finite time. From the first integral \eqref{1st Int}, we see that the phase portrait is symmetric about $n$ axis. This fact and the direction of vector fields imply that the trajectory must connect the point $(1,0)$ following a homoclinic orbit (See Figure \ref{FigTrajec}). This also shows that \textbf{(A1)} is valid since  $n_\ast < n_s$ from \eqref{g(n)=g(1)hot}. Hence a non-trivial solution $(n,E)(\xi)$ to \eqref{ODE_n_E} with \eqref{bdCon n E} exists, and it follows from the discussion on the reduction to the ODE system that a smooth non-trivial solution $(n,u,\phi)(\xi)$ to \eqref{TravelEq}--\eqref{bdCon+-inf} exists. From \eqref{TravelEqB}, one can check that $(n,u,\phi)$ also satisfies \eqref{Symmetric}--\eqref{SignOfSols} upon a shift so that \eqref{max-v} holds. This finishes the proof.
\end{proof}

\section{Peak of solitary wave solution}
We divide the proof of Proposition \ref{cor2} into four steps. In Step 1, we show that
 \begin{equation}\label{n-1 to 0}
\lim_{\veps \to 0}n_\ast^\veps = 1
\end{equation} 
as well as that $V=\sqrt{1+\sigma}$ is necessary for \eqref{n-1 to 0}. In Step 2, we obtain a rough estimate for $n_\ast^\veps-1$. This bound enable us to get a sharp estimate for $n_\ast^\veps$ in Step 3. Using this result, we obtain the sharp estimates for $u_\ast^\veps$ and $\phi_\ast^\veps$ in Step 4.

\begin{proof}[Proof of Proposition \ref{cor2}] 
\textit{Step 1}: 
We note that $1<n_c^\veps < n_\ast^\veps$ by \eqref{Aux 10} and \eqref{g(n)=g(1)hot} (or \eqref{g(n)=g(1)cold} for $\sigma=0$). By \eqref{n-1 to 0} together with \eqref{Def n_c}, we see that $V=\sqrt{1+\sigma}$ is a necessary condition for \eqref{n-1 to 0}. 

For the case $\sigma>0$, we recall from \eqref{g(n)=g(1)hot}--\eqref{g(n)=g(1)cold} that $n_\ast^\veps$ is defined as a unique root satisfying $g(n_\ast^\veps)=g(1)$ and $1 < n_\ast^\veps < n_s^\veps$. We investigate the behavior of $g$ as $\veps \to 0$ by considering the limiting case $\veps=0$. It is clear from the definition \eqref{Def_g} of $g$ that  $g(n)$ uniformly converges to a function
\[
g_0(n) := \frac{1+\sigma}{n} + \sigma n + \exp\left(\frac{1+\sigma}{2} - \frac{1+\sigma}{2n^2} - \sigma \ln n\right)
\]
as $\veps \to 0$ on $[1,L]$ for any $L>1$. We have
\begin{equation}\label{Aux Corollary1}
\frac{dg_0}{dn}(n) = -\left( \frac{1+\sigma}{n^3} -\frac{\sigma}{n} \right)\underbrace{\left[ n-\exp\left(\frac{1+\sigma}{2} - \frac{1+\sigma}{2n^2} - \sigma \ln n\right)\right]}_{=:\mathfrak{g}_0(n)} .
\end{equation}
One may easily check that $\mathfrak{g}_0(n)=0$ has a unique solution $n=1$ and that 
\begin{equation}\label{Aux Corollary}
\mathfrak{g}_0(n)>0 \quad \text{for} \quad n \in (1,\infty).
\end{equation}
Thus, for the case $\sigma>0$, $\frac{dg_0}{dn}$ vanishes only at $n=1$ and $n=n_s^0:=\sqrt{\frac{1+\sigma}{\sigma}}$, and it follows from \eqref{Aux Corollary1}--\eqref{Aux Corollary} that $g_0(n)$ strictly decreases on $(1,n_s^0)$. This implies that $g_0(n) < g_0(1)$ on $(1,n_s^0)$. Hence \eqref{n-1 to 0} holds by the construction of $n_\ast^\veps$ and the convergence of $g$ to $g_0$ as $\veps \to 0$. In a similar fashion, one can check that \eqref{n-1 to 0} holds for the case $\sigma =0$. We omit the details.

\noindent\textit{Step 2}: 
Applying Taylor's theorem to $g(n_\ast^\veps)=g(1)$ around $n=1$, we obtain 
\begin{equation}\label{AuxEp4}
\frac{g^{(2)}(1)}{2}(n_\ast^\veps-1)^2+ \frac{g^{(3)}(1)}{3!}(n_\ast^\veps-1)^3 + \frac{g^{(4)}(n^\veps_b)}{4!}(n_\ast^\veps-1)^4 = 0
\end{equation}
for some $1<n_b^\veps<n_\ast^\veps$. By the fact that $\lim_{\veps \to 0} n_\ast^\veps=1$, we find that there exist constants $\veps_{0},C>0$ and functions $\mathfrak{g}_2(\veps)$, $\mathfrak{g}_3(\veps)$ of $\veps$ such that for all $0<\veps<\veps_0$,
\begin{subequations}\label{Q2Q3}
\begin{align}[left = \empheqlbrace\,]
& g^{(2)}(1) =  2V\gamma\veps+\veps^2 \mathfrak{g}_2(\veps), \label{Q2} \\ 
& g^{(3)}(1) = -2(1+\sigma) + \veps \mathfrak{g}_3(\veps), \label{Q3} \\
& \frac{|g^{(4)}(n_b^\veps)|}{4!} + |\mathfrak{g}_2(\veps)| + |\mathfrak{g}_3(\veps)| < C \label{Q4}
\end{align}
\end{subequations}
(See Appendix for \eqref{AuxEp4} and \eqref{Q2Q3}). Dividing \eqref{AuxEp4} by $(n_\ast^\veps-1)^2$ and using \eqref{Q2}, we have
\begin{equation}\label{AuxEp5}
(n_\ast^\veps-1)\left(-\frac{g^{(3)}(1)}{3!} - \frac{g^{(4)}(n_b^\veps)}{4!}(n_\ast^\veps-1) \right) =  V\gamma\veps + \frac{\veps^2}{2}\mathfrak{g}_2.
\end{equation}
By \eqref{n-1 to 0} and \eqref{Q3}--\eqref{Q4}, we can choose sufficiently small $\veps_0>0$ (smaller than $\veps_0$ for \eqref{Q2Q3}) so that for all $0<\veps<\veps_{0}$,  
\begin{equation}\label{AuxEp6}
-\frac{g^{(3)}(1)}{3!} - \frac{g^{(4)}(n_b^\veps)}{4!}(n_\ast^\veps-1)>-\frac{g^{(3)}(1)}{3!} - \frac{1}{8} > \frac{1+\sigma}{4} - \frac{1}{8} \geq \frac{1}{8}.
\end{equation}
Dividing \eqref{AuxEp5} by the left-hand side (LHS) of \eqref{AuxEp6},  we obtain
\begin{equation}\label{Aux1 Thm2}
0<n_\ast^\veps-1 <  8V \gamma\veps + 4\veps^2\mathfrak{g}_2.
\end{equation}
\textit{Step 3}: Dividing \eqref{AuxEp5} by $-\frac{g^{(3)}(1)}{3!}$, we get
\begin{equation}\label{Aux3 Thm2}
\begin{split}
n_\ast^\veps-1 +\frac{6V\gamma}{g^{(3)}(1)}\veps
& =  - \frac{3\mathfrak{g}_2}{g^{(3)}(1)}\veps^2 - \frac{g^{(4)}(n_b^\veps)}{4g^{(3)}(1)}(n_\ast^\veps-1)^2 =: \mathfrak{g}_4.
\end{split}
\end{equation}
By \eqref{Q3}--\eqref{Q4} and \eqref{Aux1 Thm2}, we know  there is a constant $C>0$ such that 
\begin{equation}\label{C_1tilde}
|\mathfrak{g}_4| \leq \veps^2 C
\end{equation}
for all $0<\veps<\veps_0$.  Subtracting $3\gamma V^{-1}\veps$ from \eqref{Aux3 Thm2}, we obtain 
\begin{equation}\label{Aux2 Thm2}
\begin{split}
n_\ast^\veps-1- 3\gamma V^{-1}\veps
& = \frac{-6V^2-3g^{(3)}(1)}{Vg^{(3)}(1)} \gamma\veps + \mathfrak{g}_4.
\end{split}
\end{equation}
Using \eqref{Q3}--\eqref{Q4} and $V=\sqrt{1+\sigma}$, we see that there is a constant $C>0$ such that
\begin{equation}\label{C_1tilde1}
\begin{split}
\left| \frac{6V^2+3g^{(3)}(1)}{Vg^{(3)}(1)} \right| 
& < C|6V^2+3g^{(3)}(1)| = 3C\veps|\mathfrak{g}_3| < \veps C.
\end{split}
\end{equation}
Now \eqref{C_1tilde}--\eqref{C_1tilde1} imply that there is a constant $C>0$ such that for all $0<\veps<\veps_0$,
\begin{equation}\label{Aux4 Thm2}
\begin{split}
\left| n_\ast^\veps-1- 3\gamma V^{-1}\veps \right|
& \leq  \veps^2 C.
\end{split}
\end{equation}
\textit{Step 4}: 
By applying Taylor's theorem to \eqref{TravelEqB1}, we have 
\begin{equation}\label{Bd u}
\begin{split}
u_\ast^\veps  = (V+\gamma\veps)\left[(n_\ast^\veps-1) -\frac{2}{(n_b^\veps)^3}(n_\ast^\veps-1)^2  \right]
\end{split}
\end{equation}
for some $1<n_b^\veps<n_\ast^\veps$. Subtracting $3\gamma\veps$ from \eqref{Bd u} and using \eqref{Aux4 Thm2}, we obtain
\begin{equation}\label{MaxuEpsSquare}
\begin{split}
\left| u_\ast^\veps-3\gamma\veps \right|
& =  \left|  V\left(n_\ast^\veps-1-3\gamma V^{-1}\veps\right) \right.\\
& \quad \left.
-\frac{2V}{(n_b^\veps)^3} (n_\ast^\veps-1)^2  + \gamma\veps(n_\ast^\veps-1) \left(1 -\frac{2}{(n_b^\veps)^3}(n_\ast^\veps-1)  \right)  \right|\\
& \leq \veps^2C.
\end{split}
\end{equation}
Similarly, we have from \eqref{TravelEqB2} that 
\begin{equation}\label{MaxphiEpsSquare1}
\begin{split}
\phi_\ast^\veps
& = (1+2V\gamma\veps + \gamma^2\veps^2)(n_\ast^\veps-1) + \frac{dh}{dn}(n_b^\veps) (n_\ast^\veps-1)^2
\end{split}
\end{equation}
for some $1<n_b^\veps<n_\ast^\veps$. Subtracting $3\gamma V^{-1}\veps$ from \eqref{MaxphiEpsSquare1} and using \eqref{Aux4 Thm2}, we obtain
\begin{equation}\label{MaxphiEpsSquare2}
\begin{split}
\left| \phi_\ast^\veps - 3\gamma V^{-1}\veps \right|
& = \left| \left(n_\ast^\veps-1 -3\gamma V^{-1}\veps \right) \right. \\
& \quad \left. + \gamma\veps(n_\ast^\veps-1)(2V + \gamma\veps) +\frac{dh}{dn}(n_b^\veps) (n_\ast^\veps-1)^2 \right|\\
& \leq \veps^2C.
\end{split}
\end{equation}
Combining \eqref{Aux4 Thm2}, \eqref{MaxuEpsSquare} and \eqref{MaxphiEpsSquare2}, we obtain \eqref{unibd_nuphi}. This completes the proof.
\end{proof}

\section{Asymptotic behavior of solitary wave solution}
To prove Theorem \ref{MainThm3}, it is enough to show that the uniform estimate \eqref{pointestimate} holds on the interval $[0,\infty)$ thanks to symmetry of the solutions, \eqref{Symmetric}. 
\subsection{Uniform decay estimate}
Throughout this subsection, we let
\begin{equation}
\widetilde{N}_\veps(\xi):=\frac{n^\epsilon(\xi)-1}{\veps} \quad \text{and} \quad \widetilde{E}_\veps(\xi):=\frac{E^\veps(\xi)}{\veps}=\frac{-{\phi^\veps}'(\xi)}{\veps}
\end{equation}
for notational simplicity. By Theorem \ref{MainThm}, $\widetilde{N}_\veps(\xi)>0$ for $\xi\in\mathbb{R}$ and $\widetilde{E}_\veps(\xi)\geq 0$ for $\xi \in [0,\infty)$. We first prove the following preliminary lemmas,
which will be crucially used to get the uniform decay estimate for $\veps^{-1}(n^\veps-1,u^\veps,\phi^\veps)$ in Proposition \ref{Prop1}.
We note that it is enough to control only $\widetilde{N}_\veps$ in Lemma \ref{LemmaAux3}  since the right-hand side  of the nonlinear ODE \eqref{ODE_n_E} is linear in $E^\veps$. 
\begin{lemma}\label{LemmaAux1}
There exist positive constants $\veps_0$, $C$, and $\delta_0$ such that the following hold:
\begin{enumerate}
\item For all $0<\veps<\veps_0$, 
\begin{subequations}\label{secStatLem1}
\begin{align} 
& \widetilde{N}_\veps(0) > 2\gamma V^{-1}, \label{LowerboundofN} \\
& 4\gamma^2 > \widetilde{E}_\veps'(0) > 2\gamma^2 , \label{2ndDerivBdPhiat0} \\
& 1/2<\sup_{\xi \in \mathbb{R}}\, h(n^\veps)<3/2, \label{EqAppen4} \\
& \sup_{\xi\in \mathbb{R}}\left( |\widetilde{N}_\veps'(\xi)| + |\widetilde{E}_\veps(\xi)| +|\widetilde{E}_\veps''(\xi) | \right) \leq C. \label{1stDerivBdPhi}
\end{align}
\end{subequations}
\item If $0<\delta<\delta_0$ and $\widetilde{N}_\veps \leq \delta$ for all $0<\veps<\veps_0$, then there holds
\begin{equation}\label{Compatibility}
 |\widetilde{E}_\veps|> \sqrt{\gamma} \, \widetilde{N}_\veps
\end{equation}
for all $0<\veps<\veps_0$.
\end{enumerate}
Here $C$, $\delta_0$ and $\delta$ are independent of $\veps$ and $\xi$.
\end{lemma}
\begin{lemma}\label{LemmaAux3}
There exist positive constants $\veps_0$ and $\delta_1$  such that the following holds: for each $0 < \delta < \delta_1$, there exists $\xi_\delta>0$ such that
\[
0<\widetilde{N}_\veps(\xi) \leq \delta 
\]
  for all $\xi \geq \xi_\delta$ and $0<\veps<\veps_0$. Here $\delta_1$, $\delta$, and $\xi_\delta$ are independent of $\veps$ and $\xi$.
\end{lemma}

\begin{proof}[Proof of Lemma \ref{LemmaAux1}]
It is obvious that \eqref{LowerboundofN} holds for all sufficiently small $\veps$ by the sharp estimate \eqref{unibd_nuphi}. Using \eqref{unibd_nuphi} again, we find from the ODE  \eqref{ODE_n_E2} that there is a function $\tilde{\mathfrak{g}}(\veps,n_\ast^\veps)$ such that
\begin{equation}\label{EqAppen9}
\begin{split}
\widetilde{E}_\veps'(0) 
& = \frac{1}{\veps^2}\left( n^\veps(0)-e^{H(n^\veps(0))}\right) \\
& = \frac{1}{\veps^2}\left( -2V\gamma\veps(n_\ast^\veps-1) + (1+\sigma)(n_\ast^\veps-1)^2 \right) + \veps\tilde{\mathfrak{g}},
\end{split}
\end{equation}
where $\tilde{\mathfrak{g}}$ is uniformly bounded by some constant $C>0$ for all $0<\veps<\veps_0$, and thus
\[
\begin{split}
\lim_{\veps \to 0}\widetilde{E}_\veps'(0) 
& =  3\gamma^2
\end{split}
\]
(See Appendix for \eqref{EqAppen9}). 
Hence, \eqref{2ndDerivBdPhiat0} holds for all sufficiently small $\veps$. 

Since $1<n^\veps(\xi)  \leq n_\ast^\veps$ for all $\xi \in \mathbb{R}$, there exists a constant $C>0$ such that 
\begin{equation}\label{EqAppend1}
\begin{split}
\sup_{\xi \in \mathbb{R}}\left| h(n^\veps) -1 \right| 
& \leq \sup_{\xi \in \mathbb{R}}\left|(V+\gamma\veps)^2-\sigma (n^\veps)^2-(n^\veps)^3 \right| \\
& \leq (V+\gamma\veps)^2-(\sigma+1) +  \sup_{\xi \in \mathbb{R}}\left| \sigma\left(1-(n^\veps)^2\right)+\left(1-( n^\veps)^3 \right) \right|  \\
& \leq \veps C
\end{split}
\end{equation}
for all $0<\veps<\veps_0$, where we have used $V=\sqrt{1+\sigma}$ and \eqref{unibd_nuphi} in the last line. It is clear that \eqref{EqAppend1} implies \eqref{EqAppen4} for all sufficiently small $\veps$. Now we choose $\veps_0>0$ (smaller than $\veps_0$ of Proposition \ref{cor2}) so small that \eqref{LowerboundofN}-- \eqref{EqAppen4} hold.

To obtain \eqref{1stDerivBdPhi}, we first apply Taylor's theorem to the right-hand side of the first integral \eqref{1st Int} around $n=1$ and then divide the resultant by $2^{-1}\veps^3$ to get
\begin{equation}\label{EqAppen1}
\begin{split}
(\widetilde{E}_\veps)^2 
& = \frac{g^{(2)}(1)}{\veps}(\widetilde{N}_\veps)^2  + \frac{g^{(3)}(n_b^\veps)}{3}(\widetilde{N}_\veps)^3 
\end{split}
\end{equation}
for some $n_b^\veps \in(1, n^\veps)$, where 
\begin{equation}\label{EqAppen3}
g^{(2)}(1) = \gamma\veps(2V+\gamma\veps)(1+2V\gamma\veps+\gamma^2\veps^2)
\end{equation}
(See Appendix for $g^{(2)}$ and $g^{(3)}$). Thanks to \eqref{unibd_nuphi} and the fact that $1<n^\veps(\xi)\leq n_*^\veps$ for all $\xi\in\mathbb{R}$, the RHS of \eqref{EqAppen1} is bounded by some constant $C>0$ (independent of $\veps$ and $\xi$) for all $0<\veps<\veps_0$ and $\xi\in\mathbb{R}$ as long as $1<n_b^\veps <n^\veps$. Hence, we have
\begin{equation}\label{EqAppen6}
\sup_{\xi\in\mathbb{R}}|\widetilde{E}_\veps(\xi)| \leq C.
\end{equation}
Dividing the ODE \eqref{ODE_n_E1} by $\veps h$, it follows from \eqref{EqAppen6} and \eqref{EqAppen4} that 
\begin{equation}\label{EqAppen5}
\sup_{\xi \in \mathbb{R}}|\widetilde{N}_\veps'|=\sup_{\xi \in \mathbb{R}}\frac{|\widetilde{E}_\veps|}{|h(n^\veps)|}  \leq  C.
\end{equation}
Differentiating the ODE \eqref{ODE_n_E2} in $\xi$, we obtain by the mean value theorem that 
\begin{equation}\label{EqAppen8}
\begin{split}
\sup_{\xi \in \mathbb{R}} | \widetilde{E}_\veps'' | 
& =\sup_{\xi \in \mathbb{R}} \frac{1}{\veps} | \widetilde{N}_\veps'||1 - e^{H(n^\veps)}h(n^\veps)|  \\
& \leq \sup_{\xi \in \mathbb{R}} \frac{1}{\veps} | \widetilde{N}_\veps'| \left( C | n^\veps-1| + |1+\sigma-(V+\gamma\veps)^2| \right) \\
& \leq C,
\end{split}
\end{equation}
where we have used that $n^\veps>1$ in the second line, and $V=\sqrt{1+\sigma}$, \eqref{EqAppen5} and \eqref{unibd_nuphi} are used in the last line. Combining \eqref{EqAppen6}--\eqref{EqAppen8}, we obtain \eqref{1stDerivBdPhi}. 

From \eqref{EqAppen1}--\eqref{EqAppen3}, we get
\begin{equation}\label{EqAppen11}
\begin{split}
|\widetilde{E}_\veps|
 = \sqrt{\frac{g^{(2)}(1)}{\veps}}\widetilde{N}_\veps\sqrt{1 + \frac{\veps g^{(3)}(n_b^\veps)}{3g^{(2)}(1)}\widetilde{N}_\veps}  \geq  \sqrt{2\gamma}\widetilde{N}_\veps\sqrt{1 + \frac{\veps g^{(3)}(n_b^\veps)}{3g^{(2)}(1)}\widetilde{N}_\veps}
\end{split}
\end{equation}
since $V\geq 1$ for $\sigma \geq 0$ and $\widetilde{N}_\veps(\xi)>0$ for $\xi \in \mathbb{R}$. Now, by choosing $\delta_0>0$ so that 
\begin{equation}\label{EqAppen12}
\delta_0\sup_{\xi \in \mathbb{R}}\sup_{0<\veps<\veps_0}\left|\frac{\veps g^{(3)}(n_b^\veps)}{3g^{(2)}(1)} \right| <\frac{1}{2},
\end{equation}
\eqref{Compatibility} follows from \eqref{EqAppen11}. In \eqref{EqAppen12}, the supremum exists by \eqref{EqAppen3} and the estimate \eqref{unibd_nuphi}.
This completes the proof.
\end{proof}

\begin{proof}[Proof of Lemma \ref{LemmaAux3}]
From Theorem \ref{MainThm}, we recall that for each $\veps$, $\widetilde{N}_\veps(\xi)$ strictly decreases on $[0,\infty)$ and $\lim_{\xi \to \infty}\widetilde{N}_\veps(\xi)=0$.  This together with \eqref{LowerboundofN}, by the intermediate value theorem, implies that   for each $0<\veps<\veps_0$ and $0<\delta <  \gamma V^{-1}$,  there exists a unique $\xi_{\veps,\delta}>0$ such that 
\begin{equation}\label{Eq9Lem4}
\widetilde{N}_\veps(\xi_{\veps,\delta})=\delta.
\end{equation}
By \eqref{Compatibility} and \eqref{Eq9Lem4}, we have that for all $0<\veps<\veps_0$ and $0 < \delta < \min\{\delta_0,\gamma V^{-1}\}$, 
\begin{equation}\label{Eq4Lem4}
\widetilde{E}_\veps (\xi_{\veps,\delta})>\sqrt{\gamma}\delta
\end{equation}
since $\widetilde{E}_\veps(\xi)\geq 0$ on $[0,\infty)$. On the other hand, by the mean value theorem, we get
\[
\widetilde{E}_\veps'(\xi) = \widetilde{E}_\veps'(0) - \widetilde{E}_\veps''(\xi^\veps_b) \xi
\]
for some $\xi^\veps_b\in (0,\xi)$. Thus, using \eqref{2ndDerivBdPhiat0} and \eqref{1stDerivBdPhi}, we find that there exists a sufficiently small $\xi_0>0$ (independent of $\veps$) such that for all $0<\veps<\veps_0$ and $0 < \xi <  \xi_0$, there holds
\begin{equation}\label{Eq1Lem4}
\widetilde{E}_\veps'(\xi) > \gamma^2.
\end{equation}
  Integrating \eqref{Eq1Lem4} over $[0,\xi]$ for $0<\xi < \xi_0$, we obtain 
\begin{equation}\label{Eq11Lem4}
\widetilde{E}_\veps(\xi) > \gamma^2\xi
\end{equation}
since $\widetilde{E}_\veps(0)=0$. Hence, for all $\delta>0$ satisfying $\gamma^{-\frac{3}{2}}\delta < \xi_0$, we have 
\begin{equation}\label{Eq3Lem4}
\widetilde{E}_\veps(\gamma^{-\frac{3}{2}}\delta) > \sqrt{\gamma}\delta.
\end{equation}

Let us assume for a moment that there is a positive number $\delta_1 \leq \min\{\delta_0,\gamma V^{-1},\gamma^{\frac{3}{2}}\xi_0\}$ such that for all $0<\delta < \delta_1$ and $0<\veps<\veps_0$, 
\begin{equation}\label{Eq2Lem4}
\gamma^{-\frac{3}{2}}\delta < \xi_{\veps,\delta}
\end{equation}
holds. Then \eqref{Eq4Lem4} and \eqref{Eq3Lem4}--\eqref{Eq2Lem4} imply that
\begin{equation}\label{Eq5Lem4}
\widetilde{E}_\veps(\xi)>\sqrt{\gamma}\delta
\end{equation}
for all $\xi\in[\gamma^{-\frac{3}{2}}\delta, \xi_{\veps,\delta}]$ and $0<\veps<\veps_0$ (See Figure \ref{FigTrajec}).
 Applying \eqref{EqAppen4} and \eqref{Eq5Lem4} to the ODE \eqref{ODE_n_E1}, we have 
\begin{equation}\label{Eq6Lem4}
\begin{split}
-\widetilde{N}_\veps'(\xi) 
& = \frac{\widetilde{E}_\veps(\xi)}{h(n_\veps)}  \geq  \frac{2}{3}\sqrt{\gamma}\delta
\end{split}
\end{equation}
for $\xi\in[\gamma^{-\frac{3}{2}}\delta, \xi_{\veps,\delta}]$ and $0<\veps<\veps_0$. Integrating \eqref{Eq6Lem4} from $\gamma^{-\frac{3}{2}}\delta$ to $\xi_{\veps,\delta}$, we obtain
\begin{equation}\label{Eq7Lem4}
\begin{split}
0<\widetilde{N}_\veps(\xi_{\veps,\delta})
& \leq - \frac{2}{3}\sqrt{\gamma}\delta(\xi_{\veps,\delta}- \gamma^{-\frac{3}{2}}\delta) + \sup_{0<\veps<\veps_0}\widetilde{N}_\veps(\gamma^{-\frac{3}{2}}\delta).
\end{split}
\end{equation}

Now it is clear that $\xi_{\delta}:=\sup_{0<\veps<\veps_0}\xi_{\veps,\delta}<\infty$ for each $0<\delta<\delta_1$. If not, for some sequence $\{\veps_k\}$, the right-hand side of \eqref{Eq7Lem4} tends to $-\infty$ as $k \to \infty$, and this is a contradiction. Therefore, we obtain that for each $0<\delta<\delta_1$,
\[
\delta =\widetilde{N}_\veps(\xi_{\veps,\delta}) \geq   \widetilde{N}_\veps(\xi_\delta) \geq \widetilde{N}_\veps(\xi)
\]
for all $0<\veps<\veps_0$ and $\xi \geq \xi_\delta$.

 To complete the proof, it is remained to verify \eqref{Eq2Lem4}. It suffices to show that there exists $\delta_1 \leq \min\{\delta_0,\gamma V^{-1}, \gamma^{\frac{3}{2}}\xi_0\}$ such that 
\begin{equation}\label{Eq10Lem4}
\gamma^{-\frac{3}{2}}\delta < \inf_{0<\veps<\veps_0}\xi_{\veps,\delta}
\end{equation}
for all $0<\delta<\delta_1$.
It is obvious that the infimum exists and finite thanks to $0<\xi_{\veps,\delta}<\infty$.
 Suppose, to the contrary, that $\inf_{0<\veps<\veps_0}\xi_{\veps,\delta}=0$ for some $0<\delta \leq \gamma V^{-1}$.   
Then there is a sequence $\{\veps_k\}$ such that $\xi_{\veps_k,\delta} \to 0$ as $k \to \infty$. On the other hand, by the mean value theorem, \eqref{LowerboundofN} and \eqref{Eq9Lem4} imply that 
\begin{equation}\label{Eq8Lem4}
-\widetilde{N}_{\veps_k}'(\bar{\xi}_{\veps_k,\delta}) =\frac{\widetilde{N}_{\veps_k}(0)-\widetilde{N}_{\veps_k}(\xi_{\veps_k,\delta})}{\xi_{\veps_k,\delta}} \geq \frac{\frac{2\gamma}{V}-\delta}{\xi_{\veps_k,\delta}}>0
\end{equation}
for some $0<\bar{\xi}_{\veps_k,\delta}<\xi_{\veps_k,\delta}$. This is a contradiction since the right-hand side of \eqref{Eq8Lem4} tends to $+\infty$ as $k \to \infty$ while the left-hand side of \eqref{Eq8Lem4} stays bounded thanks to \eqref{1stDerivBdPhi}. Hence, we   have $\inf_{0<\veps<\veps_0}\xi_{\veps,\delta}>0$ for all $0<\delta \leq\gamma V^{-1}$. By the definition, $\xi_{\veps,\delta}$ is a decreasing function of $\delta$ for each fixed $\veps$, and this implies that 
\begin{equation}\label{222}
\inf_{0 < \veps<\veps_0}\xi_{\veps,\delta} \geq \inf_{0 < \veps<\veps_0}\xi_{\veps,\frac{\gamma}{V}}>0
\end{equation}
for all $0 < \delta \leq \gamma V^{-1}$. By letting $\delta_1:=\min\{\delta_0,\gamma V^{-1},\gamma^{\frac{3}{2}}\xi_0,\gamma^{\frac{3}{2}}\inf_{0<\veps<\veps_0}\xi_{\veps,\frac{\gamma}{V}}\}$, \eqref{222} implies that \eqref{Eq10Lem4} holds for all $0<\delta<\delta_1$.
This completes the proof.
\end{proof}
Now using the previous lemmas, we  show the uniform exponential decay of the solutions, which plays an important role in the analysis of remainders. 
\begin{proposition}\label{Prop1}
Let $k$ be any non-negative interger. Then there exist constants $C_\ast,\veps_1>0$ (independent of $k$), and $C_k>0$ such that 
\begin{subequations}\label{UnifExpDecayn}
\begin{align}[left = \empheqlbrace\,]
& |{\widetilde{N}_\veps}^{(k)}(\xi)|  +  |{\widetilde{E}_\veps}^{(k)}(\xi)| \leq C_k e^{-C_\ast \xi},  \label{UnifExpDecayn1} \\
& \frac{|{u^\veps}^{(k)}(\xi)|}{\veps}  +  \frac{|{\phi^\veps}^{(k)}(\xi)|}{\veps} \leq C_k e^{-C_\ast\xi} \label{UnifExpDecayn2}
\end{align}
\end{subequations}
for all $0<\veps < \veps_1$ and $\xi\geq 0$. Here $C_k$ and $C_\ast$ are independent of $\veps$ and $\xi$.
\end{proposition}

\begin{proof}
By Taylor's theorem, the ODE system \eqref{ODE_n_E} can be written as
\begin{equation}\label{ExpandODE}
\begin{pmatrix}
n^\veps-1\\
E^\veps
\end{pmatrix}'
= A\begin{pmatrix}
n^\veps-1 \\
E^\veps
\end{pmatrix} 
+
\begin{pmatrix}
\mathcal{R}_1 \\
\mathcal{R}_2
\end{pmatrix},
\end{equation}
where
\begin{equation}
A:=\begin{pmatrix}
0 & (\sigma-J^2)^{-1} \\
\veps^{-1}(1 +\sigma - J^2) & 0
\end{pmatrix}
\end{equation}
is the Jacobian matrix of \eqref{ODE_n_E} at $(n^\veps,E^\veps)=(1,0)$ (See \eqref{jacobian}). Here, $\mathcal{R}_1$ and $\mathcal{R}_2$ are functions of $(n^\veps,E^\veps)$, and there is a constant $C>0$ such that 
\begin{subequations}\label{R1R2estimat}
\begin{align}[left = \empheqlbrace\,]
& |\mathcal{R}_1| \leq C \left( (n^\veps-1)^2 + (n^\veps-1)E^\veps \right),  \label{R1R2estimate1} \\
& |\mathcal{R}_2| \leq \veps^{-1}C(n^\veps-1)^2\label{R1R2estimate2}
\end{align}
\end{subequations}
for all $0 < \veps < \veps_0$ by the estimate \eqref{unibd_nuphi} (See Appendix). The eigenvalues of $A$ are 
\begin{equation}\label{EigenValOfA}
\pm\frac{1}{\sqrt{\veps}}\sqrt{\frac{J^2-1-\sigma}{J^2-\sigma}} 
= \pm \sqrt{\frac{ 2V\gamma + \gamma^2\veps}{1+2V\gamma\veps + (\gamma\veps)^2 }},
\end{equation}
where we have used the definition \eqref{J} of $J$. Let $\lambda=\lambda(\veps)$ be the positive eigenvalue of $A$. By \eqref{EigenValOfA}, one can choose  $\veps_\ast=\veps_\ast(\sigma,\gamma)>0$ sufficiently small so that 
\[
\frac{d\lambda}{d\veps}(\veps)= \frac{1}{2\lambda(\veps)} \frac{\gamma^2(1-4V^2)+O(\veps)}{(1+2V\gamma\veps+\gamma^2\veps^2)^2} < 0
\]
for all $0<\veps<\veps_\ast$ since $V=\sqrt{\sigma +1}$. 
Thus $\lambda(\veps)$ is a decreasing function of $\veps$ on $(0,\veps_\ast)$. We fix such $\veps_\ast$.  On the other hand, one can check from \eqref{EigenValOfA} that $(J^2(\veps) -\sigma)\lambda(\veps)$ is an increasing function of $\veps$. Thus we have for all $0<\veps<\veps_\ast$,
\begin{subequations}\label{BdEigenValue}
\begin{align}[left = \empheqlbrace\,]
& -\lambda < -\lambda(\veps_\ast),  \label{BdEigenValue1} \\
& \sqrt{2V\gamma} < (J^2-\sigma)\lambda  < (J^2(\veps_\ast)-\sigma)\lambda(\veps_\ast).\label{BdEigenvalue2}
\end{align}
\end{subequations}

By diagonalization, we can rewrite \eqref{ExpandODE} as
\begin{equation}\label{ODEDiago}
P\begin{pmatrix}
n^\veps-1 \\
E^\veps
\end{pmatrix}'
= \begin{pmatrix}
\lambda & 0 \\
0 & -\lambda
\end{pmatrix}
P\begin{pmatrix}
n^\veps-1 \\
E^\veps
\end{pmatrix} + P\begin{pmatrix}
\mathcal{R}_1 \\
\mathcal{R}_2
\end{pmatrix},
\end{equation}
where 
\[
P :=\frac{1}{2} \begin{pmatrix}
1 & -[(J^2-\sigma)\lambda]^{-1} \\
1 & [(J^2-\sigma)\lambda]^{-1}
\end{pmatrix}.
\]
Multiplying the second component of \eqref{ODEDiago} by $2\veps^{-1}$, and then using \eqref{R1R2estimat}, we get
\begin{equation}\label{RemainderExpand2}
\begin{split}
\widetilde{N}_\veps' + [(J^2-\sigma)\lambda]^{-1}\widetilde{E}_\veps' 
& \leq  -\lambda\underbrace{\left( \widetilde{N}_\veps + [(J^2-\sigma)\lambda]^{-1}\widetilde{E}_\veps \right)}_{=:I_1} \\
& \quad + \veps C \widetilde{N}_\veps^2 + \veps C \widetilde{N}_\veps \widetilde{E}_\veps +  C[(J^2-\sigma)\lambda]^{-1} \widetilde{N}_\veps^2.
 \end{split}
\end{equation}
Since $\widetilde{N}_\veps$ and $\widetilde{E}_\veps$ are non-negative for $\xi \in [0,\infty)$, we obtain from \eqref{BdEigenValue} that
\begin{equation}\label{I1}
\begin{split}
-\lambda \, I_1 
& \leq -\lambda(\veps_\ast)I_1 \\
& \leq -\frac{\lambda(\veps_\ast)}{2}I_1 \underbrace{- \frac{\lambda(\veps_\ast)}{2}\left( \widetilde{N}_\veps + \left[(J^2(\veps_\ast)-\sigma)\lambda(\veps_\ast) \right]^{-1}\widetilde{E}_\veps \right)}_{=:I_2}
\end{split}
\end{equation}
By Lemma \ref{LemmaAux3} and \eqref{BdEigenvalue2}, one can choose a sufficiently small $\delta \in(0,\delta_1)$ so that
\begin{equation}\label{I2}
I_2 +  \veps C \widetilde{N}_\veps^2 + \veps C \widetilde{N}_\veps \widetilde{E}_\veps +C[(J^2-\sigma)\lambda]^{-1} \widetilde{N}_\veps^2 < \frac{1}{2} I_2 <0
\end{equation}
holds for all $0 < \veps < \veps_0$ and $\xi\geq \xi_\delta$.
Let $\veps_1:=\min\{\veps_{\ast},\veps_0\}$. Combining \eqref{RemainderExpand2}--\eqref{I2},
  we get
\begin{equation}\label{RemainderExpand3}
\begin{split}
\widetilde{N}_\veps' + [(J^2-\sigma)\lambda]^{-1}\widetilde{E}_\veps' 
& < -\frac{\lambda(\veps_\ast)}{2}\left( \widetilde{N}_\veps + [(J^2-\sigma)\lambda]^{-1}\widetilde{E}_\veps \right).
\end{split}
\end{equation}
We multiply \eqref{RemainderExpand3} by $e^{2^{-1}\lambda(\veps_\ast)\xi}$ and then integrate it over $[\xi_\delta, \xi]$ to obtain
\begin{equation}\label{RemainderExpand4}
\begin{split}
\widetilde{N}_\veps(\xi) + \frac{\widetilde{E}_\veps(\xi)}{(J^2-\sigma)\lambda}
&  \leq \left[ \widetilde{N}_\veps(\xi_\delta) + \frac{\widetilde{E}_\veps(\xi_\delta)}{(J^2-\sigma)\lambda} \right]e^{-\frac{\lambda(\veps_\ast)}{2}(\xi-\xi_\delta)}.
\end{split}
\end{equation}
Let $C_\ast:=2^{-1}\lambda(\veps_\ast)$. Now \eqref{UnifExpDecayn1} for $k=0$ immediately follows from \eqref{RemainderExpand4} by applying the bounds \eqref{1stDerivBdPhi}, \eqref{BdEigenvalue2}, and the estimate \eqref{unibd_nuphi}. \eqref{UnifExpDecayn2} is obtained from \eqref{UnifExpDecayn1} and \eqref{TravelEqB}. This completes the proof for $k=0$. By induction, we obtain the cases $k\geq  1$ using the system \eqref{ODE_n_E} (See \eqref{Ap44} for instance). We omit the details.
\end{proof}
 
\subsection{The remainder equations - Proof of Theorem \ref{MainThm3}}
We derive the remainder equations for $(n_R^\veps,u_R^\veps,\phi_R^\veps)$:
\begin{subequations}\label{Eq for difference}
\begin{align}[left = \empheqlbrace\qquad]
& u_R^\veps-Vn_R^\veps = \mathcal{M}_1^\veps, \label{Mas2a} \\ 
& \phi_R^\veps-Vu_R^\veps+\sigma n_R^\veps = \mathcal{M}_2^\veps, \label{Mom2a} \\
& \phi_R^\veps-n_R^\veps= V\mathcal{M}_1^\veps + \mathcal{M}_2^\veps, \label{VxMas2a+Mom2a}
\end{align}
\end{subequations}
and
\begin{equation}\label{MainEq}
{\phi_{R}^\veps}'' - F_\veps\phi_R^\veps = \mathcal{M}_3^\veps,
\end{equation}
where 
\begin{equation}\label{F}
F_\veps(\xi) :=   2V\gamma  - 2V^2 n_{\text{KdV}} - V^2\frac{\phi_R^\veps}{\veps}
\end{equation}
and $\mathcal{M}_i^\veps$ $(i=1,2,3)$, defined in \eqref{Mas1}--\eqref{Mom1} and \eqref{M3Def}, are some functions   of $n_{\text{KdV}}, n^\veps_{R}, u^\veps_{R}$, and $\phi^\veps_{R}$. For notational simplicity, we let $n_{\text{KdV}}=n_{K}$. Since $n_{K}(\xi) = n_K(x-\gamma t)$ satisfies the associated  KdV equation \eqref{KdV}, it also satisfies 
\begin{equation}\label{KdVinXi}
-\gamma n_{K}' + V n_{K}n_{K}' + (2V)^{-1}n_{K}''' = 0.
\end{equation} 

Putting \eqref{Def nR uR phiR} into \eqref{TravelEqA1}, a direct calculation yields \eqref{Mas2a}, where
\begin{equation}\label{Mas1}
\begin{split}
\mathcal{M}_1^\veps
& :=  (\gamma \veps - \veps Vn_{K}) n_R^{\veps} - (\veps n_{K} + n_R^{\veps})u_R^{\veps}   + \veps^2(\gamma n_{K} - Vn_{K}^2) \\
& \quad + \veps\underbrace{ ( V n_{K} -  Vn_{K} ) }_{=0}.
\end{split}
\end{equation}
Similarly, we obtain \eqref{Mom2a} from \eqref{Def nR uR phiR} and \eqref{TravelEqA2}, where
\begin{equation}\label{Mom1} 
\begin{split}
\mathcal{M}_2^\veps
& :=  (\gamma \veps  -  \veps Vn_K)u_R^{\veps} - \frac{|u_{R}^{\veps}|^2}{2}  + \frac{\sigma}{2}\left( 2\veps n_K n_R^{\veps} + |n_R^{\veps}|^2\right) -\sigma O_{n^{\veps}}(\veps^3) \\
& \quad  + \veps^2\left(\gamma Vn_K - \frac{V^2n_K^2}{2} + \frac{\sigma n_K^2}{2}\right)  - \veps\underbrace{(n_K-V^2n_K+\sigma n_K)}_{=0}
\end{split}
\end{equation}
and $O_{n^{\veps}}(\veps^3) := \ln n^{\veps} - (n^{\veps}-1)+\frac{1}{2}(n^{\veps}-1)^2$.
On the other hand,   \eqref{VxMas2a+Mom2a} follows from \eqref{Mas2a}--\eqref{Mom2a} since $V=\sqrt{1+\sigma}$. 

Now we derive \eqref{MainEq}-\eqref{F}. Plugging \eqref{Def nR uR phiR} into \eqref{TravelEq3}, we obtain
\begin{equation}\label{Poi1}
\begin{split}
n_{R}^{\veps} - \phi_R^{\veps}
& = -\veps^2n_K'' - \veps {\phi_{R}^{\veps}}''   +\frac{1}{2}(\veps n_K + \phi_{R}^{\veps})^2 + O_{\phi^\veps}(\veps^3) ,
\end{split}
\end{equation}
where $O_{\phi^\veps}(\veps^3) :=  e^{\phi^\veps} - 1 - \phi^\veps - \frac{1}{2}(\phi^\veps)^2$. 
By adding \eqref{Mom2a} and \eqref{Poi1}, the term $\phi_{R}^{\veps}$ in the left-hand side of \eqref{Poi1} is canceled, and we have
\begin{equation}\label{Mom+Poi1}
\begin{split}
-Vu_R^{\veps} + (1+\sigma)n_R^{\veps} 
& = \veps^2\left( -n_K'' + \frac{n_K^2}{2} \right) - \veps {\phi_{R}^{\veps}}''  + \frac{1}{2}\left( 2\veps n_K\phi_R^{\veps} + |\phi_R^{\veps}|^2 \right)\\
& \quad  + O_{\phi^\veps}(\veps^3) + \mathcal{M}_2^\veps.
\end{split}
\end{equation}
Multiplying \eqref{Mas2a} by $V$, and then adding the resulting equation to \eqref{Mom+Poi1}, the left-hand side of \eqref{Mom+Poi1} is canceled since $V=\sqrt{1+\sigma}$. Hence we get
\begin{equation}\label{R1}
\begin{split}
0 & = V\mathcal{M}_1^\veps + \{\text{the RHS of \eqref{Mom+Poi1}} \}\\
& = V\left[(\gamma\veps -\veps Vn_K) n_{R}^{\veps} - (\veps n_K + n_{R}^{\veps})u_{R}^{\veps} \right]  - \veps{\phi_R^{\veps}}'' + \frac{1}{2}\left( 2\veps n_K\phi_R^{\veps} + |\phi_R^{\veps}|^2 \right)  \\
& \quad +(\gamma\veps  - \veps Vn_K )u_R^{\veps} - \frac{|u_R^{\veps}|^2}{2} + \frac{\sigma}{2}\left(2\veps n_K n_R^{\veps} + |n_R^{\veps}|^2  \right)  -\sigma O_{n^\veps}(\veps^3)  \\
& \quad + O_{\phi^\veps}(\veps^3) + \veps^2 \left( \underline{ 2V\gamma n_K - V^2n_K^2  -\frac{V^2n_K^2}{2} + \frac{1+\sigma}{2}n_K^2 -n_K'' } \right).
\end{split}
\end{equation}
Since $V=\sqrt{1+\sigma}$ and $n_K$ satisfies \eqref{KdVinXi}, the underlined part of \eqref{R1} is zero. Using \eqref{Mas2a}, we substitute $u_R^{\veps}$ of \eqref{R1} with $Vn_R^{\veps}+\mathcal{M}_1^\veps$. Then it is straightforward to obtain
\begin{equation}\label{R2}
\begin{split}
& \veps{\phi_R^{\veps}}''+ \left(\veps\widetilde{F} + \frac{1}{2}(3V^2-\sigma)n_R^\veps \right) n_R^{\veps} - \frac{1}{2}\left( 2\veps n_K + \phi_R^{\veps} \right)\phi_R^{\veps} =   \widetilde{\mathcal{M}}_3^\veps,
\end{split}
\end{equation}
where
\begin{subequations}
\begin{align}[left = \empheqlbrace\,]
 \widetilde{F}(\xi)
& :=-2V\gamma +(3V^2-\sigma) n_K , \label{F3} \\
 \begin{split}
\widetilde{\mathcal{M}}_3^\veps
& :=  \mathcal{M}_1^\veps \left( \gamma\veps-2V\veps n_K-2Vn_R^\veps-\frac{1}{2}\mathcal{M}_1^\veps \right) -\sigma O_{n^\veps}(\veps^3) + O_{\phi^\veps}(\veps^3).
\end{split}
 \label{M3}
\end{align}
\end{subequations}
Using \eqref{VxMas2a+Mom2a}, we substitute $n_R^{\veps}$ of the left-hand side of \eqref{R2} by $\phi_R^{\veps} + V\mathcal{M}_1^\veps + \mathcal{M}_2^\veps$ and then divide the resultant by $\veps$.  Then we get
\begin{equation}\label{R3}
{\phi_R^{\veps}}''  -\left[2V\gamma - (3V^2-\sigma-1)n_K - \frac{1}{2}(3V^2-\sigma-1)\frac{\phi_R^\veps}{\veps}\right]\phi_R^\veps = \mathcal{M}_3^\veps,
\end{equation}
where 
\begin{equation}\label{M3Def}
\begin{split} 
 \mathcal{M}_3^\veps
&:=   \frac{\widetilde{\mathcal{M}}_3^\veps}{\veps} - \left(V\mathcal{M}_1^\veps + \mathcal{M}_2^\veps\right) \left[ \widetilde{F}  + \frac{(3V^2-\sigma)}{2\veps}\left[2\phi_R^\veps + \left(V\mathcal{M}_1^\veps + \mathcal{M}_2^\veps\right) \right] \right].
\end{split}
\end{equation}
Substituting $V=\sqrt{1+\sigma}$ into \eqref{R3}, we finally arrive at  \eqref{MainEq}--\eqref{F}.

The following lemma is a direct consequence of  Proposition~\ref{cor2}, Proposition~\ref{Prop1}, and the definitions of $n_{\text{KdV}},n_R^\veps,u_R^\veps,\phi_R^\veps,F_\veps$, and $\mathcal{M}_i^\veps$ $(i=1,2,3)$. We omit the proof.
\begin{lemma}\label{Lemma6} 
Let $k$ be any non-negative integer. Then there exist constants $\widetilde{C}_\ast$, $\veps_1>0$ (independent of $k$), and $C_k$, $\xi_1>0$ such that for all $0<\veps<\veps_1$,
\begin{subequations}
\begin{align}[left = \empheqlbrace\qquad]
& |{n_R^\veps}^{(k)}|, \; |{u_R^\veps}^{(k)}|, \; |{\phi_R^\veps}^{(k)}| \leq C_k\veps e^{-\widetilde{C}_\ast\xi}, \quad (\xi \geq 0),\label{Order nR uR phiR} \\
& |{M_i^\veps}^{(k)}| \leq C_k\veps^2 e^{-\widetilde{C}_\ast\xi}, \quad (\xi \geq 0, \quad  i=1,2,3),\label{Decay M} \\
& \sup_{\xi \in [0,\infty)}|{F_\veps}^{(k)}| \leq C_k, \label{Bd F} \\
& F_\veps(\xi)>V\gamma, \quad (\xi \geq \xi_1).  \label{LowerBd F} 
\end{align}
\end{subequations}
Here $C_k$ and $\widetilde{C}_\ast$ are independent of $\veps$ and $\xi$, and $\xi_1$ is independent of $\veps$.
\end{lemma}
Now we shall show, by a continuation argument, that the remainders near the peak decay in $\veps$ at the same order as the peak does. 
\begin{proposition}\label{Prop2}
For any fixed $\xi_\ast>0$, there exists a constant $C_{\xi_\ast}>0$ such that
\begin{equation}\label{EqProp2}
\begin{split}
\sup_{\xi\in[0,\xi_\ast]} \left( |{\phi_R^\veps}' (\xi) |^2 + | \phi_R^\veps(\xi)|^2 \right) \leq C_{\xi_\ast}\veps^4
\end{split}
\end{equation}
holds for all $0< \veps <\veps_1$. Here $C_{\xi_\ast}$ is independent of $\veps$ but depends on $\xi_\ast$.
\end{proposition}
\begin{proof}
Multiplying \eqref{MainEq} by $2{\phi_R^\veps}'$, and then adding $(|\phi_R^\veps|^2)' = 2\phi_R^\veps{\phi_R^\veps}'$ to the resulting equation, we have
\begin{equation}\label{R4}
\left( |{\phi_R^\veps}'|^2 + |\phi_R^\veps|^2 \right)' =  2(1+F_\veps) \phi_R^\veps{\phi_R^\veps}' + 2\mathcal{M}_3^\veps{\phi_R^\veps}'.
\end{equation}
Applying Young's inequality to \eqref{R4}, there exists a constant $C>0$ such that 
\begin{equation}\label{R5}
\begin{split}
\left( |{\phi_R^\veps}'|^2 + |\phi_R^\veps|^2 \right)' 
& \leq   |1+F_\veps|\left(|\phi_R^\veps|^2 +|{\phi_R^\veps}'|^2 \right) +  |\mathcal{M}_3^\veps|^2 +  |{\phi_R^\veps}'|^2 \\
& \leq C\left( |{\phi_R^\veps}'|^2 + |\phi_R^\veps|^2 \right) + C\veps^4
\end{split}
\end{equation}
for $\xi \in [0,\infty)$, where we have used \eqref{Decay M}--\eqref{Bd F} in the last inequality. Multiplying \eqref{R5} by $e^{-C\xi}$, and then integrating over $[0,\xi]$, we obtain
\begin{equation}\label{R6}
\begin{split}
|{\phi_R^\veps}'|^2(\xi) +|\phi_R^\veps|^2(\xi)
& \leq  \left( |{\phi_R^\veps}'|^2(0)  +|\phi_R^\veps|^2(0) \right) e^{C\xi} + \veps^4(e^{C\xi}-1) \\
& \leq \veps^4C^2e^{C\xi} +  \veps^4(e^{C\xi}-1),
\end{split}
\end{equation}
where we have used the estimate \eqref{unibd_nuphi} and that ${\phi_R^\veps}'(0)=0$ since ${\phi^\veps}'(0)=n_{\text{KdV}}'(0) = 0$. By taking the supremum of the left-hand side of \eqref{R6} over $[0,\xi_\ast]$ for any fixed $\xi_\ast>0$, the proof is done.
\end{proof}
Next we shall investigate the uniform decay of the remainders in the far-field.
\begin{proposition}\label{Prop3}
There exist constants $\veps_1,\xi_1,C,\alpha>0$ such that
\begin{equation}\label{EqProp3}
\begin{split}
\int_{\xi_1}^\infty\left( |{\phi_R^\veps}' (\xi) |^2 + | \phi_R^\veps(\xi)|^2 \right)e^{\alpha\xi}\, d\xi \leq  C\veps^4
\end{split}
\end{equation}
for all $0<\veps<\veps_1$. Here $\alpha$ is independent of $\veps$ and $\xi$.
\end{proposition}
\begin{proof}
Multiplying \eqref{MainEq} by $\phi_R^\veps e^{\alpha\xi}$ for $0<\alpha<\widetilde{C}_\ast$, and then integrating the resultant from $\xi_1$ to $\infty$, where $\widetilde{C}_\ast$ and $\xi_1$  are the constants of Lemma \ref{Lemma6}, we have
\begin{equation}\label{Dia3}
\begin{split}
\int_{\xi_1}^\infty \left[ F_\veps(\xi)|\phi_R^\veps|^2 +|{\phi_R^\veps}'|^2 \right] e^{\alpha\xi} \,d\xi
 = & \int_{\xi_1}^\infty-\mathcal{M}_3^\veps\phi_R^\veps e^{\alpha\xi}  -\alpha {\phi_R^\veps}'\phi_R^\veps e^{\alpha\xi} \,d\xi \\
& - ({\phi_R^\veps}'\phi_R^\veps e^{\alpha\xi})(\xi_1),
\end{split}
\end{equation}
where we have used the fact that $\lim_{\xi \to \infty}{\phi_R^\veps}'\phi_R^\veps e^{\alpha \xi} = 0$ for all $\veps >0$, which is true from \eqref{Order nR uR phiR} since $0<\alpha<\widetilde{C}_\ast$. From \eqref{LowerBd F}, we know that
\begin{equation}\label{Dia6}
\int_{\xi_1}^\infty \left[ V\gamma|\phi_R^\veps|^2 +|{\phi_R^\veps}'|^2\right]e^{\alpha\xi} \,d\xi < \text{the LHS of \eqref{Dia3}}.
\end{equation} 
Applying Young's inequality, and then using \eqref{Decay M} and Proposition \ref{Prop2}, we see that there is a constant $C_{\xi_1}>0$ such that 
\begin{equation}\label{Dia4}
\begin{split}
\text{the RHS of \eqref{Dia3}}
& \leq \frac{V\gamma+\alpha}{2}\int_{\xi_1}^\infty |\phi_R^\veps|^2e^{\alpha\xi}\,d\xi + \frac{1}{2V\gamma}\int_{\xi_1}^\infty |\mathcal{M}_3^\veps|^2 e^{\alpha\xi}\,d\xi \\
& \quad + \frac{\alpha}{2}\int_{\xi_1}^\infty |{\phi_R^\veps}'|^2e^{\alpha\xi}\,d\xi + |({\phi_R^\veps}'\phi_R^\veps e^{\alpha\xi})|(\xi_1) \\
& \leq \frac{V\gamma+\alpha}{2}\int_{\xi_1}^\infty |\phi_R^\veps|^2e^{\alpha\xi}\,d\xi + \frac{\alpha}{2}\int_{\xi_1}^\infty |{\phi_R^\veps}'|^2e^{\alpha\xi}\,d\xi \\
&\quad  + \veps^4 C_{\xi_1,\alpha}.
\end{split}
\end{equation}
We choose $0<\alpha<\min\{\widetilde{C}_\ast,2,V\gamma\}$. By \eqref{Dia6}--\eqref{Dia4}, the proof is done.
\end{proof}

\begin{proof}[Proof of Theorem \ref{MainThm3}]
We first show that for all non-negative integers $k$, there exists constant $C_k>0$ such that 
\begin{equation}\label{MainThm31}
\int_0^\infty |{\phi_R^\veps}^{(k)}(\xi)|^2 e^{\alpha\xi}\,d\xi \leq C_k \veps^4
\end{equation}
for all $0<\veps<\veps_1$.
We prove by induction. Suppose that for $k=0,1,\cdots,n$, there exists a constant $C_k>0$ such that \eqref{MainThm31} holds.
Taking the $n$-th derivative of \eqref{MainEq} in $\xi$, and then multiplying the resultant by ${\phi_R^\veps}^{(n+2)}e^{\alpha\xi}$, we have
\[
\begin{split}
\int_0^\infty |{\phi_R^\veps}^{(n+2)}|^2e^{\alpha\xi}\,d\xi
&  = \int_0^\infty\left[(F_\veps\phi_R^\veps)^{(n)} + {\mathcal{M}_3^\veps}^{(n)}\right]{\phi_R^\veps}^{(n+2)}e^{\alpha\xi}\,d\xi \\
& \leq \frac{1}{2}\int_0^\infty|{\phi_R^\veps}^{(n+2)}|^2e^{\alpha\xi}\,d\xi +  \int_0^\infty |(F_\veps\phi_R^\veps)^{(n)}|^2 e^{\alpha\xi}\,d\xi \\
&\quad  + \int_0^\infty|{\mathcal{M}_3^\veps}^{(n)}|^2 e^{\alpha\xi}\,d\xi\\ 
& \leq \frac{1}{2}\int_0^\infty|{\phi_R^\veps}^{(n+2)}|^2e^{\alpha\xi}\,d\xi +   \sum_{i=0}^nC_i\int_0^\infty |{\phi_R^\veps}^{(i)}|^2e^{\alpha\xi}\,d\xi\\
& \quad + C_n \veps^4 \\
& \leq \frac{1}{2}\int_0^\infty|{\phi_R^\veps}^{(n+2)}|^2e^{\alpha\xi}\,d\xi + C_n \veps^4,
\end{split}
\]
where we have used Young's inequality in the first inequality, \eqref{Decay M}--\eqref{Bd F} in the second inequality, and the induction hypothesis in the last inequality.
Hence, \eqref{MainThm31} holds for $k=n+2$. From Proposition \ref{Prop2}--\ref{Prop3}, we know that \eqref{MainThm31} is true for $k=0,1$, which finishes the proof of \eqref{MainThm31}.

By the fundamental theorem of calculus and the Cauchy-Schwarz inequality, we see that \eqref{MainThm31} implies that for each $k$,
\begin{equation}\label{MainThm32}
\begin{split}
\sup_{\xi \in [0,\infty)}|{\phi_R^\veps}^{(k)}(\xi)e^{\frac{\alpha}{2}\xi}|^2 
& \leq 2\| {\phi_R^\veps}^{(k)} \|_{L^2_\alpha}\left(\| {\phi_R^\veps}^{(k+1)}\|_{L^2_\alpha}+\frac{\alpha}{2}\| {\phi_R^\veps}^{(k)}\|_{L^2_\alpha} \right) \\
& \leq C_k \veps^4,
\end{split}
\end{equation}
where $\|\cdot\|_{L^2_\alpha}:=\|\cdot e^{\frac{\alpha}{2}\xi}\|_{L^2([0,\infty))}$. Now the estimates for $n_R^\veps$ and $u_R^\veps$ follow from \eqref{Eq for difference} by applying \eqref{MainThm32} and \eqref{Decay M}. Since $(n_R^\veps,u_R^\veps,\phi_R^\veps)$ is symmetric about $\xi=0$, the proof of Theorem \ref{Main Result} is done.

\end{proof}


\section{Appendix}

\subsection{Solutions of the equation \eqref{Aux3 lem-1}}
Obviously, $z=1$ is a solution to \eqref{Aux3 lem-1}. Taking the logarithm on both side of \eqref{Aux3 lem-1}, we observe that \eqref{Aux3 lem-1} is equivalent to 
\[
0=\sigma \ln z + \ln \sigma + \ln \left[(z-1)^2 + \sigma^{-1}\right]-\frac{\sigma}{2}\left(z^2-1 \right)=:f(z).
\]
By a direct calculation, we obtain
\[
\begin{split}
\frac{df}{dz}(z)
& = \frac{-\sigma(z-1)^2}{z\left[(z-1)^2 + \sigma^{-1} \right]}\left(z-\sqrt{\frac{1+\sigma}{\sigma}} \right)\left(z+\sqrt{\frac{1+\sigma}{\sigma}} \right).
\end{split}
\]
Since $f(1)=0$ and $\lim_{z \to +\infty}f(z) = -\infty$, there is a unique $\zeta_{\sigma}>\sqrt{\frac{1 + \sigma}{\sigma}}>1$ such that $f(\zeta_{\sigma})=0$ and $f(z) > 0$ for $1<z < \zeta_{\sigma}.$ This also implies that \eqref{Aux4-1 lem} holds.


\subsection{Derivation of \eqref{AuxEp4}--\eqref{Q2Q3} and \eqref{EqAppen1}--\eqref{EqAppen3}} For simplicity, let $'$ denote the derivative in $n$. Taking the $k$-th derivatives of $g(n)$ in $n$, we have
\[
\left\{
\begin{array}{l l}
g^{(1)}(n) = -J^2n^{-2} + \sigma + e^{H}h , \quad g^{(2)}(n) = 2J^2n^{-3} + e^{H}h^2 + e^{H}h', \\
g^{(3)}(n) = -6J^2n^{-4} + e^{H}h^3 + 3e^{H}hh' + e^{H}h'', \\
g^{(4)}(n) = 24J^2n^{-5} + e^{H}h^4 + 6e^{H}h^2h' + 3e^{H}(h')^2   + 4e^{H}hh''+e^{H}h'''.
\end{array} 
\right.
\]
From the definitions \eqref{Def_H} and \eqref{Def_h} of $H$ and $h$, we have 
\begin{equation}\label{H derivatives}
H(1)=0, \quad h(1)=J^2-\sigma, \quad h'(1)=-3J^2+\sigma, \quad h''(1)=12J^2-2\sigma,
\end{equation}
and thus we obtain $g^{(1)}(1)=0$, $g^{(2)}(1) =  (J^2-\sigma)(J^2-1-\sigma)$ and
\[
g^{(3)}(1) = 6J^2-2\sigma + (J^2-\sigma)^3 + 3(J^2-\sigma)(-3J^2+\sigma).
\]
By the definition \eqref{J} of $J$, a direct calculation yields \eqref{Q2}--\eqref{Q3}. By the above calculation and the definitions of $H$, $h$ and $J$, \eqref{Q4} is clear owing to the estimate \eqref{unibd_nuphi}. In a similar fashion, \eqref{EqAppen1} and \eqref{EqAppen3} follow.

\subsection{Derivation of \eqref{EqAppen9}} Let $q(n):= n-e^{H(n)}$. A direct calculation yields that
\[
q^{(1)}(n)=1-e^{H}h, \quad q^{(2)}(n)=(-h^2 -  h')e^H, \quad q^{(3)}=(-h^3-3hh'-h'')e^H.
\] 
Using \eqref{H derivatives} and the definition \eqref{J} of $J$, we obtain from the Taylor theorem that
\begin{equation}\label{Ap44}
\begin{split}
\frac{1}{\veps^2}(n^\veps-e^{H(n^\veps)}) 
& = \frac{1}{\veps^2}\left[ (1+\sigma-J^2)(n^\veps-1) \right. \\
& \quad \left. - \frac{1}{2}((J^2-\sigma)^2-3J^2+\sigma)(n^\veps-1)^2   + g^{(3)}(n_b^\veps)(n^\veps-1)^3 \right]\\
& = \left( -2V\gamma + O(\veps) \right)\widetilde{N}_\veps  +\left( 1+\sigma +O(\veps)\right)\widetilde{N}_\veps^2 +\veps g^{(3)}(n_b^\veps)\widetilde{N}_\veps^3
\end{split}
\end{equation}
for some $1<n_b^\veps<n^\veps$. By \eqref{unibd_nuphi}, we obtain \eqref{EqAppen9}.
\subsection{Derivation of \eqref{R1R2estimat}} It is enough to find the second-order partial derivatives of $\mathcal{R}_1$ and $\mathcal{R}_2$. By a direct calculation, we have   $\partial^2_E\mathcal{R}_1=\partial_{nE}\mathcal{R}_2 = \partial_E^2 \mathcal{R}_2 = 0$ and
\[
\left\{
\begin{array}{l l}
\partial_{n}^2\mathcal{R}_1 = Eh^{-2}h'' - 2Eh^{-3}(h')^2 , \quad \partial_{nE} \mathcal{R}_1 = h'h^{-2}, \\
\partial_{n}^2\mathcal{R}_2 = -\veps^{-1}e^{H}(h^2 +h' ),
\end{array}
\right.
\]
Hence, \eqref{R1R2estimat} follows from the definitions of $h$ and $H$ the estimate \eqref{unibd_nuphi}.


  \section*{Acknowledgments.}
B. K. was supported by Basic Science Research Program through the National Research Foundation of Korea(NRF) funded by the Ministry of science, ICT and future planning(NRF-2017R1E1A1A03070406).

\bibliographystyle{amsplain}




 \end{document}